\documentclass{amsart}
\usepackage[T1]{fontenc}
\usepackage{amssymb,mathtools,microtype,enumitem,xurl,aliascnt}
\usepackage[hidelinks,pdfusetitle,linktoc=all]{hyperref}
\usepackage[capitalise,nameinlink,noabbrev]{cleveref}
\newtheorem{theorem}{Theorem}[section]
\newaliascnt{lemma}{theorem}
\newtheorem{lemma}[lemma]{Lemma}
\aliascntresetthe{lemma}
\newaliascnt{proposition}{theorem}
\newtheorem{proposition}[proposition]{Proposition}
\aliascntresetthe{proposition}
\newaliascnt{corollary}{theorem}
\newtheorem{corollary}[corollary]{Corollary}
\aliascntresetthe{corollary}
\newaliascnt{claim}{theorem}

\aliascntresetthe{claim}
\theoremstyle{definition}
\newaliascnt{definition}{theorem}

\aliascntresetthe{definition}
\theoremstyle{remark}
\newaliascnt{remark}{theorem}
\newtheorem{remark}[remark]{Remark}
\aliascntresetthe{remark}
\numberwithin{equation}{section}
\newcommand{\Q}{\mathbb Q}
\newcommand{\R}{\mathbb R}
\newcommand{\Z}{\mathbb Z}
\newcommand{\N}{\mathbb N}

\newcommand{\PP}{\mathbb P}
\newcommand{\OO}{\mathcal O}
\newcommand{\floor}[1]{\left\lfloor#1\right\rfloor}
\newcommand{\ceil}[1]{\left\lceil#1\right\rceil}

\newcommand{\num}{\equiv}

\DeclareMathOperator{\Supp}{Supp}
\DeclareMathOperator{\Exc}{Exc}
\DeclareMathOperator{\Diff}{Diff}
\DeclareMathOperator{\coeff}{coeff}
\DeclareMathOperator{\Div}{div}

\DeclareMathOperator{\Spec}{Spec}

\DeclareMathOperator{\lct}{lct}
\DeclareMathOperator{\ord}{ord}

\setlist{leftmargin=*,itemsep=2pt,topsep=4pt}

\title[Fibered Fano threefolds]{Boundedness of complements for fibered Fano threefolds in positive characteristic}
\author{Xintong Jiang}
\hypersetup{pdftitle={Boundedness of complements for fibered Fano threefolds in positive characteristic},pdfauthor={Xintong Jiang}}
\date{}
\begin{document}
\begin{abstract}
Let $(X,B)$ be a log canonical threefold pair with hyperstandard coefficients and a contraction to a curve or a surface. If $X$ is of Fano type over the base and $-(K_X+B)$ is relatively nef, we obtain a bounded monotonic complement near each prescribed fibre in characteristic greater than five. For projective $X$ of Fano type with $-(K_X+B)$ nef and $K_X+B$ rationally trivial over the base, we obtain a global complement in sufficiently large characteristic. The bounds depend only on the finite set defining the coefficients. This also treats the case in which $-(K_X+B)$ is nef, numerically nonzero and not big. The proof uses a genus-zero canonical bundle formula allowing a residual horizontal cycle of geometric weighted degree less than one.
\end{abstract}
\maketitle
\enlargethispage{4pt}
\tableofcontents

\section{Introduction}\label{sec:introduction}

Let $X$ be a smooth Fano variety. If $|-K_X|$ contains a smooth divisor $D$, adjunction gives $K_D\sim0$, and already on $X$ we have $K_X+D\sim0$. The theory of complements asks for a version of this construction on a singular variety with a prescribed boundary $B$. One seeks an effective divisor to add to $B$ so that the new pair is log canonical and a fixed multiple of its log canonical divisor is principal. The difficulty is to choose the multiple independently of $X$ and $B$. Birkar proved such bounds for Fano-type varieties in characteristic zero \cite[Theorems~1.7 and~1.8]{BirkarAnti}. Here we work in positive characteristic and use the presence of a fibration to pass from dimension three to dimension two or one.

Fix a finite set $R\subset[0,1]\cap\Q$. We allow coefficients in the \emph{hyperstandard set}
\[
 \Phi(R)=\{0,1\}\cup\{1-r/m:r\in R,\ m\in\Z_{>0}\}.
\]
Thus the denominators of the input boundary can be arbitrarily large. Write $N=-(K_X+B)$. The term \emph{monotonic $n$-complement} means an enlargement $B^+\geq B$ for which $(X,B^+)$ is log canonical, $n(K_X+B^+)\sim0$, and the rounding inequality in \eqref{pre:f-complement} holds. A \emph{contraction} $f:X\to Z$ is a projective surjection with $f_*\OO_X=\OO_Z$. A variety is \emph{of Fano type over $Z$} if some effective rational boundary $\Gamma$ makes $(X,\Gamma)$ Kawamata log terminal and $-(K_X+\Gamma)$ ample over $Z$. The boundary $\Gamma$ can differ from $B$.

\begin{theorem}[Global complements for fibered threefolds]
\label{global:assembly}\label{main:global}\label{main:47}
There are integers $n_{\mathrm{fib}}(R)>0$ and $p_{\mathrm{fib}}(R)>5$ with the following property. Let $(X,B)$ be a projective log canonical threefold pair over an algebraically closed field of characteristic $p>p_{\mathrm{fib}}(R)$. Suppose that $X$ is of Fano type, that $\coeff B\subset\Phi(R)$ and that $N$ is nef. If there is a contraction $f:X\to Z$ with $0<\dim Z<3$ and $K_X+B\sim_{\Q,Z}0$, then $(X,B)$ has a global monotonic $n_{\mathrm{fib}}(R)$-complement.
\end{theorem}

\begin{corollary}\label{fib:nonbig}
With the same $R$, $p_{\mathrm{fib}}(R)$ and $n_{\mathrm{fib}}(R)$, let $(X,B)$ be a projective log canonical threefold pair as above, with $X$ of Fano type and $\coeff B\subset\Phi(R)$. If $N$ is nef, numerically nonzero and not big, then $(X,B)$ has a global monotonic $n_{\mathrm{fib}}(R)$-complement.
\end{corollary}

In the corollary the fibration is supplied by a multiple of $N$. There is also a local statement in which a fibration is given and $K_X+B$ need only be anti-nef along its fibres.

\begin{theorem}[Complements near a prescribed fibre]
\label{bridge:relative-complements}\label{main:relative}\label{main:45}\label{main:514}
There is an integer $n_{\mathrm{rel}}(R)>0$ with the following property in every characteristic $p>5$. Let $(X,B)$ be a log canonical threefold pair over an algebraically closed field and $f:X\to Z$ a contraction with $0<\dim Z<3$. Suppose that $X$ is of Fano type over $Z$, that $\coeff B\subset\Phi(R)$ and that $N$ is nef over $Z$. For every $z\in Z$ there is a monotonic $n_{\mathrm{rel}}(R)$-complement on an open neighbourhood of the fibre over $z$.
\end{theorem}

\paragraph{\textbf{How the proofs fit together}}
The local theorem starts with a surface $S$ on a birational model of $X$ that meets the prescribed fibre. We arrange that $S$ occurs with coefficient one. Adjunction then gives a surface pair with controlled coefficients. A complement on that pair gives a section on $S$. To extend it to $X$, it is enough to make an explicit higher direct image vanish. The sheaf in this calculation is associated to an integral Weil divisor, whose Cartier index need not be bounded. \Cref{sec:surface-complements} constructs the surface complements; \cref{sec:relative-complements} proves the required vanishing and carries out the lifting.

For the global theorem, we instead put the log canonical divisor on the base in the form
\[
 K_X+B\sim_\Q f^*(K_Z+B_Z+M_Z).
\]
The coefficient of $B_Z$ along a divisor records how much of its pullback can be added before log canonicity fails. The remaining term $M_Z$ records information about the generic fibre. We need its behaviour on every birational model of $Z$, and we need a bounded multiple for which the formula holds up to an integral principal divisor. Then a complement downstairs pulls back at the same index.

For a surface base, the geometric generic fibre is a marked projective line. In \cref{cbf:anchored} we allow marked points that split over a Galois extension of degree prime to $p$, together with an effective residual cycle of geometric weighted degree strictly less than one. The moduli divisor is semiample on a suitable birational model, and its traces above this model are pullbacks. The example in \cref{sec:wild-counterexample} corrects Theorem~3.4(3) of \href{https://arxiv.org/abs/2506.00553v2}{the previous version}. Its general log fibre is klt, yet wild ramification gives a moduli trace with negative degree on a curve. In the application to Fano type, the original marks split tamely, and we choose a perturbation for which the additional weight is less than one. For the original boundary, \cref{moduli:effective-surface-base} gives a finite rational set $I=I(R)$ with $\coeff B_Z\subset\Phi(I)$ and an integer $q=q(R)$ such that the moduli part has a nef Cartier multiple $q\mathbf M$ on a birational model and $q(K_X+B)\sim qf^*(K_Z+B_Z+M_Z)$. Over a curve, boundedness of the surface fibres yields geometric log canonicity in large characteristic. We can then apply Benozzo's curve-base formula \cite[Theorem~7.10 and Corollary~7.11]{Benozzo}. Both arguments and the lifting to $X$ are in \cref{sec:fibered-complements}.

\subsection*{Acknowledgments}
The author thanks Caucher Birkar for his guidance and suggestions, and Zheng Xu, Fulin Xu, Xiaowei Jiang, Haoran Zhu and Florin Ambro for comments.

\paragraph{\textbf{Disclosure of AI usage}}
The human author developed the main content of this paper and collaborated with GPT-5.6 Sol and GPT-6.1 Sol to audit and revise this version. GPT-6.2 Astra organized the revised material into the final manuscript.

\section{Preliminaries}\label{sec:preliminaries}

We work over an algebraically closed field $k$ of characteristic $p>5$, unless another field is specified. Varieties are integral and quasi-projective. In a statement about a generic fibre the ground field is $k(Z)$; the geometric generic fibre is taken over $\overline{k(Z)}$. We follow the terminology for pairs in \cite[Ch.~2]{KM}.

\subsection{Divisors and complements}
We write $\sim$, $\sim_\Q$ and $\equiv$ for linear, rational linear and numerical equivalence. A rational divisor is $\Q$-Cartier if some positive multiple is Cartier. On a $\Q$-factorial variety this holds for every Weil divisor. For an integral Weil divisor $D$ we use the divisorial sheaf $\OO_X(D)$, extending the corresponding line bundle from the regular locus. For a contraction $f:X\to Z$, $D\sim_{\Q,Z}0$ means $D\sim_\Q f^*L$ for a $\Q$-Cartier divisor on $Z$, locally on $Z$ when the statement is local. Nefness over $Z$ is tested on curves contracted by $f$. For a nef divisor $N$ on a projective $d$-fold, $\nu(N)=\max\{i:N^i\cdot H^{d-i}>0\}$ for any ample divisor $H$; the values $0$ and $d$ correspond to numerical triviality and bigness, respectively.

For a pair $(X,B)$ we require $X$ normal, $B\geq0$ and $K_X+B$ $\Q$-Cartier. Unless otherwise specified, boundaries have rational coefficients in $[0,1]$. Write $\{B\}=B-\floor B$. A monotonic $n$-complement is an lc boundary $B^+\geq B$ satisfying
\begin{equation}\label{pre:f-complement}
 n(K_X+B^+)\sim0,\qquad
 nB^+\geq n\floor B+\floor{(n+1)\{B\}} .
\end{equation}
Over a point $z\in Z$ these conditions hold on $f^{-1}(U)$ for some open neighbourhood $U$ of $z$. The first condition requires an integral principal divisor. For $0\le b<1$, $\floor{(n+1)b}\leq\ceil{nb}$, so monotonicity and integrality of $nB^+$ imply the second condition. In particular we may replace an index by any of its multiples.

We write $\coeff B$ for the set of coefficients of $B$. DCC and ACC stand for the descending and ascending chain conditions, respectively. For finite $R$, the set $\Phi(R)$ is DCC and its intersection with $[0,1-\eta]$ is finite for every $\eta>0$. An ACC theorem for thresholds rules out a strictly increasing sequence tending to one from below; this is how a uniform rounding threshold enters the proofs.

\subsection{Singularities and adjunction}
For a prime divisor $E$ on a birational model $\mu:Y\to X$, its log discrepancy is $a(E,X,B)=1-\coeff_E B_Y$, where $K_Y+B_Y=\mu^*(K_X+B)$. The pair is log canonical (lc) if all these numbers are nonnegative, Kawamata log terminal (klt) if they are positive, and $\epsilon$-lc if they are at least $\epsilon$. The image of a divisor of log discrepancy zero is an lc centre. Crepant transforms may have negative coefficients; the same inequalities then define sub-lc and sub-klt. An lc pair with $K_X+B\sim_\Q0$ is log Calabi--Yau. For an lc pair and a nonzero effective $\Q$-Cartier divisor $P$, the log canonical threshold $\lct(X,B;P)$ is the largest $t$ for which $(X,B+tP)$ is lc. On a log resolution it is the minimum of $a(E,X,B)/\ord_E P$ over divisors with positive denominator.

A divisorial log terminal (dlt) pair is lc and is simple normal crossings at the generic points of all its lc centres. A purely log terminal (plt) pair is lc and all exceptional log discrepancies are positive. A crepant dlt modification displays lc places as divisors of coefficient one on a $\Q$-factorial model. If $S$ is a normal such divisor, divisorial adjunction gives the different on $S$:
\[
 (K_X+B)|_S=K_S+\Diff_S(B-S).
\]
In the simple normal crossings case the different is $(B-S)|_S$. Inversion of adjunction allows one to test log canonicity near $S$ on the adjoint surface pair.

\subsection{Birational operations}
An MMP for $D$ contracts $D$-negative extremal rays, using a flip when the contraction is small. With scaling, one follows the nef divisors $D+tA$ as $t$ decreases between consecutive contractions. A Mori fibre space has lower-dimensional base, relative Picard number one and relatively ample negative log canonical divisor. We use the threefold MMP in the forms of \cite{BirkarCharP,BW,GNT}. For semiampleness we use \cite[Corollary~1.5]{ABL}: if $(X,\Gamma)$ is a klt threefold, $L$ is nef Cartier and $aL-(K_X+\Gamma)$ is nef and big over $Z$ for some $a>0$, a multiple of $L$ is free over $Z$ provided $\dim Z>0$ or $\nu(L)>0$ in the absolute case. On a variety of Fano type we may take $\Gamma$ to be a log Fano boundary.

\subsection{The discriminant and the moduli divisor}
Suppose $K_X+B\sim_\Q f^*L$ for a contraction $f:X\to Z$. For a prime divisor $P\subset Z$ set
\[
 t_P=\sup\{t:(X,B+t f^*P)\text{ is sub-lc over }\eta_P\}.
\]
The \emph{discriminant} is $B_Z=\sum_P(1-t_P)P$, and the \emph{moduli part} is $M_Z=L-K_Z-B_Z$. For instance, if near the generic point of a fibre component the total space is smooth, the boundary coefficient on the component is $b$, and the multiplicity of $f^*P$ there is $m$, that component imposes $t_P\leq(1-b)/m$. One must test exceptional divisors as well to get the exact threshold. In the product $\PP^1\times Z$ with $Z$ smooth and the two disjoint sections $0,\infty$ of coefficient one, $K_X+B\sim f^*K_Z$. Every prime on $Z$ has threshold one, and the discriminant and moduli parts can both be taken to be zero.

Repeat the threshold construction on each normal birational model $Z'\to Z$, using crepant pullback on a compatible model of $X$. The resulting traces $M_{Z'}$ are compatible under proper pushforward. Such a collection of divisors is a \emph{b-divisor}, written $\mathbf M$. It is b-$\Q$-Cartier if on some model it is rational Cartier and on every higher model its trace is the pullback from that model. It is b-nef or b-semiample when, on a model on which it descends in this sense, the corresponding divisor is nef or semiample. Fixing a rational trivialization on the generic fibre fixes the principal-divisor ambiguity in the moduli part.

A \emph{generalized pair} $(Z,B_Z+\mathbf M)$ keeps the nef b-divisor as part of the data. Its log divisor on $Z$ is $K_Z+B_Z+M_Z$. To compute discrepancies, on a higher model $h:Z'\to Z$ one uses
\[
 K_{Z'}+B_{Z'}+M_{Z'}=h^*(K_Z+B_Z+M_Z).
\]
The log discrepancy at $P\subset Z'$ is $1-\coeff_P B_{Z'}$. A generalized complement increases the boundary while keeping $\mathbf M$ fixed. Its index $n$ is chosen so that $n\mathbf M$ is b-Cartier; then \eqref{pre:f-complement} uses $n(K_Z+B_Z^++M_Z)$. This is what allows a complement to be pulled back with its original integer index.

On a geometric generic fibre $\PP^1$ a log Calabi--Yau boundary has weighted degree two: $\deg K_{\PP^1}=-2$. In the calculation over the function field, a closed point of degree $d$ with coefficient $b$ contributes $bd$, including inseparable multiplicity. If the points split after a Galois extension of degree prime to $p$, divisors over a base discrete valuation ring have tame ramification. This permits comparison of log discrepancies across the cover. For the remaining cycle in \cref{cbf:anchored}, the proof uses its degree bound directly over the valuation ring.

\section{Surface complements}
\label{sec:surface-complements}

The threefold arguments require complements on both ordinary and generalized surfaces. We first collect the adjunction and rounding calculations, then prove the global surface theorem with the nef b-divisor fixed. The final subsection treats surfaces over a curve; its complement is defined on a neighbourhood of a prescribed fibre.

\subsection{Adjunction}

\begin{theorem}
\label{a-thm:div-adjunction}
Let $(X,B)$ be a pair, let $S$ be a component of $\lfloor B\rfloor$, and let $\nu\colon S^\nu\to S$ be the normalization.  There is a canonically determined effective divisor $B_{S^\nu}$ such that
\[
  K_{S^\nu}+B_{S^\nu}
  \sim_{\R} \nu^*((K_X+B)|_S).
\]
If $B$ is a $\Q$-divisor, this is a $\Q$-linear equivalence.  Let $\Phi\subset[0,1]$ be a DCC set of rational numbers and assume that $(X,B)$ is lc outside a closed subset of codimension at least three and that $\coeff(B)\subset\Phi$.  Then $B_{S^\nu}$ is a boundary and
\[
 \coeff(B_{S^\nu})\subset
 \mathcal S_\Phi:=
 \{1\}\cup
 \left\{
  \frac{m-1}{m}+\frac1m\sum_i l_i b_i\le 1
  \ \middle|\ m\in\N_{>0},\ l_i\in\N,
  b_i\in\Phi
 \right\}.
\]
Moreover, for every finite $R\subset[0,1]\cap\Q$ there is a finite set $R'$, depending only on $R$, such that
\[
  \mathcal S_{\Phi(R)}\subset\Phi(R').
\]
\end{theorem}

\begin{proof}
Adjunction and the coefficient formula are \cite[Sections~4.1--4.2]{BirkarCharP}. For the last assertion, ignore zero summands in $S=\sum l_ib_i\leq1$. Every summand $1-r/n$ with $n\geq2$ is at least $1/2$. If two occur, then $S=1$. If exactly one occurs and there are no other positive summands, the adjunction coefficient is $1-r/(mn)\in\Phi(R)$. Otherwise the positive summands from the finite set $\{1-r:r\in R\}$ bound $n$ and all their multiplicities. The same multiplicity bound applies when no denominator $n\geq2$ occurs. Thus, apart from coefficients already in $\Phi(R)$, only finitely many sums $S$ arise. Adjoining their complements $1-S$ and zero to $R$ gives the required $R'$.
\end{proof}

\medskip
We will also use inversion of adjunction in dimension three.

\begin{theorem}
\label{a-thm:ioa}
Let $(X,S+B)$ be a three-dimensional pair over a field of characteristic $p>5$, where $S+B\ge0$, $\lfloor S+B\rfloor=S$, and $S$ is a prime Weil divisor.  Write
\[
 K_{S^\nu}+B_{S^\nu}=(K_X+S+B)|_{S^\nu}.
\]
Then:
\begin{enumerate}
 \item $(X,S+B)$ is lc in a neighbourhood of $S$ if and only if $(S^\nu,B_{S^\nu})$ is lc;
 \item if $X$ is $\Q$-factorial, then $(X,S+B)$ is plt in a neighbourhood of $S$ if and only if $(S^\nu,B_{S^\nu})$ is klt.  In the plt case, $S$ is normal.
\end{enumerate}
\end{theorem}
\begin{proof}
The lc equivalence is \cite[Corollary~10.1]{BMP}. For the plt assertion, apply lc inversion and take a dlt modification. If the pair were not plt near $S$, connectedness of the exceptional fibre would give a coefficient-one exceptional divisor meeting the transform of $S$. Adjunction would then contradict klt of the different. Normality in the plt case is \cite[Lemma~5.2]{BirkarCharP}; the reverse implication is plt adjunction. For real boundaries, rational approximation in the lc coefficient polytope gives the same conclusion, with the approximations chosen in the klt-adjoint locus for the plt assertion.
\end{proof}

\subsection{Transfer and complement arithmetic}

The following comparison transfers a complement without changing its index. The nef b-divisor is the same on the two models.

\begin{lemma}\label{r-lem:subcrepant}
Let $(X_1,B_1+M_1)$ and $(X_2,B_2+M_2)$ be generalized pairs over the same base $T$, with the nef parts induced by one b-divisor.  Let $g_i\colon W\to X_i$ be a common model on which that b-divisor descends, and suppose
\begin{equation}
 g_1^*(K_{X_1}+B_1+M_1)+P
   =g_2^*(K_{X_2}+B_2+M_2),\qquad P\geq0.                \label{eq:display-4-1-1}
\end{equation}
If $(X_2,B_2+M_2)$ has a monotonic $n$-complement over a neighbourhood of $t\in T$, then $(X_1,B_1+M_1)$ has a monotonic $n$-complement there as well.
\end{lemma}

\begin{proof}
After shrinking $T$, let $B_2^+$ be the given complement and put
\[
 F=P+g_2^*(B_2^+-B_2)\geq0,\qquad B_1^+=B_1+g_{1*}F.
\]
Then
\[
 g_1^*(K_{X_1}+B_1+M_1)+F
       =g_2^*(K_{X_2}+B_2^++M_2).
\]
Pull back the principal $n$-fold complemented divisor to $W$ and push it to $X_1$. Thus $n(K_{X_1}+B_1^++M_1)$ is principal. The difference between its pullback and $n$ times the displayed divisor is exceptional and numerically trivial over $X_1$, hence zero by the negativity lemma. This proves generalized lc and linear triviality. Effectivity of $F$ gives $B_1^+\geq B_1$. Since $n\mathbf M$ is b-Cartier, $nB_1^+$ is integral; monotonicity and integrality imply the complement inequality.
\end{proof}

\medskip
The rounding calculation needed at the end of this comparison is particularly simple for hyperstandard coefficients. If a seed element is written in lowest terms as $r/\ell$, its hyperstandard coefficients have the form $1-r/(\ell n)$ used below.

\label{s-sec:rounding}

\begin{lemma}
\label{s-lem:rounding}
Let
\[
 b=1-\frac{r}{\ell n}\in[0,1),
 \qquad 0<\frac r\ell\le1,
 \qquad r,\ell,n\in\Z_{>0},
\]
and let $m>0$ be an integer divisible by $\ell$.  Then
\[
 \lceil mb\rceil=\lfloor(m+1)b\rfloor.
\]
Consequently, if every coefficient of a boundary $B<1$ has this form and $m$ is divisible by the denominators occurring in the finite seed set, then
\[
 \lceil mB\rceil=\lfloor(m+1)B\rfloor
\]
coefficientwise.
\end{lemma}

\begin{proof}
Write $m=q\ell$ and $x=qr/n$.  Since $m$ is integral,
\[
 \lceil mb\rceil=m-\lfloor x\rfloor,
 \qquad
 \lfloor(m+1)b\rfloor=m+1-
   \left\lceil x+\frac{r}{\ell n}\right\rceil.
\]
Write $qr=an+h$ with $0\le h\le n-1$.  Then
\[
 x+\frac{r}{\ell n}
 =a+\frac{h+r/\ell}{n},
 \qquad 0<h+r/\ell\le n.
\]
Hence $\lceil x+r/(\ell n)\rceil=a+1=\lfloor x\rfloor+1$, which gives the equality.  The divisor statement follows coefficientwise.
\end{proof}

\begin{proposition}
\label{s-prop:complement-arithmetic}
Let $(X,S+B)$ be a pair with $S$ reduced and $B<1$, and let $m$ satisfy Lemma~\ref{s-lem:rounding}.  Put
\[
 L:=-m(K_X+S)-\lfloor(m+1)B\rfloor
\]
and assume only that $L$ is an integral Weil divisor on the neighbourhood in question.  If $\Delta\ge0$ is an integral Weil divisor with $\Delta\sim L$ (equivalently, $\Delta$ is the zero divisor of a nonzero section of the divisorial sheaf $\OO_X(L)$), define
\[
 B^+:=\frac1m\bigl(\lfloor(m+1)B\rfloor+\Delta\bigr),
 \qquad \Theta^+:=S+B^+.
\]
Then
\[
 m(K_X+\Theta^+)\sim0,
 \qquad
 m\Theta^+\ge mS+\lfloor(m+1)B\rfloor,
 \qquad \Theta^+\ge S+B.
\]
In particular, if $(X,\Theta^+)$ is lc, then $K_X+\Theta^+$ is an $m$-complement of $K_X+S+B$.
\end{proposition}

\begin{proof}
By definition, $m(K_X+S)+\floor{(m+1)B}+\Delta\sim0$. The remaining assertions follow from $\Delta\geq0$ and \cref{s-lem:rounding}, which gives $\floor{(m+1)B}=\ceil{mB}\geq mB$.
\end{proof}

\subsection{Ordinary surface preparations}

\label{v-sec:exceptionalsurface}

We first work with ordinary projective surface pairs. The rounding lemma gives finite boundary coefficients, after which a complement on a rational curve can be lifted through a plt surface pair.

\medskip
We first round coefficients sufficiently close to one.

\begin{lemma}
\label{v-lem:surface-rounding}
Fix a DCC set $\Phi\subset[0,1]\cap\Q$.  There is a rational $\eta=\eta(\Phi)>0$ with the following property, uniformly for $p>5$. Let $V$ be a projective surface of Fano type and write
\[
 C=M+\sum_j c_jE_j,\qquad C^\dagger=M+\sum_jE_j,
 \qquad 1-\eta\leq c_j\leq1,
\]
where the $E_j$ are distinct prime divisors not contained in $\Supp M$ and the coefficients of $M$ belong to $\Phi$. If $(V,C)$ is lc and $-(K_V+C)$ is nef, then $(V,C^\dagger)$ is lc and $-(K_V+C^\dagger)$ is pseudo-effective.  Its MMP gives a birational morphism $r\colon V\to W$, where $W$ is of Fano type, $(W,C_W^\dagger)$ is lc, $-(K_W+C_W^\dagger)$ is nef, and
\begin{equation}
 r^*(K_W+C_W^\dagger)\geq K_V+C^\dagger.
 \label{v-eq:surface-round-comparison}
\end{equation}
The coefficients $c_j$ need not themselves belong to a fixed DCC set.
\end{lemma}
\begin{proof}
This is \cref{gen:rounding} with nef part zero. Its proof below uses only surface threshold and global ACC, together with the surface MMP; in particular, it is independent of the surface complement theorems. The MMP negativity comparison gives \eqref{v-eq:surface-round-comparison}.
\end{proof}

\begin{remark}
\label[remark]{v-rem:finite-reduction}
For $\Phi=\Phi(I)$, round precisely those coefficients in $[1-\eta,1]$.  The remaining set $J=(\Phi(I)\cap[0,1-\eta))\cup\{0,1\}$ is finite. The resulting nef model has coefficients in $J$, and \eqref{v-eq:surface-round-comparison} is also an inequality above $K_V+B$.  Hence every monotonic $n$-complement on that model pulls back crepantly to a monotonic $n$-complement on $V$, without enlarging the index.  Effectivity follows from the inequality; Cartierness and linear triviality follow by pulling back the linearly trivial Cartier divisor $n(K+B^+)$.  Monotonicity and integrality imply the rounding condition, since $\floor{(n+1)b}\leq\ceil{nb}$ for $0\leq b<1$.
\end{remark}

\medskip
We next complement a rational curve and prepare a plt surface centre. The lifting construction will be given for generalized pairs in \cref{gen:plt-lifting}; taking the nef part zero gives the ordinary case.

\label{v-sec:nonexceptional}

For the lifting arguments, the coefficient set is finite. The reduction from hyperstandard coefficients is \cref{v-rem:finite-reduction}.

\begin{lemma}
\label{v-lem:curve-complement}
For a finite rational set $I$ and an integer $q>0$, there is an integer $n=n(I,q)$ divisible by $q$ such that every lc boundary $D$ on $\mathbb P^1$ with coefficients in $\Phi(I)$ and $\deg D\leq2$ has a monotonic $n$-complement.  The integer is characteristic-independent.
\end{lemma}
\begin{proof}
Adjoin zero and one to $\Phi(I)$ and call the resulting compact set $A$.  Its elements are rational, its only possible accumulation point is one, and its positive elements have a positive minimum $\mu$. Every boundary has at most $r=\floor{2/\mu}$ components.  Padding by zeros, its coefficients lie in the compact set
\[
 \mathcal K=\{(a_1,\ldots,a_r)\in A^r:\ \textstyle\sum a_j\leq2\}.
\]
Each $a\in\mathcal K$ has a relative neighbourhood in which $b_j\leq a_j$ for all $j$: coordinates below one are isolated in $A$, and one is an upper bound.  Choose a finite subcover and let $n$ clear the coordinates of its centres as well as $q$. Raise the given coefficient vector to one such centre.  Add distinct new points with coefficient $1/n$ until the degree is two.  The resulting lc boundary $D^+$ is monotonic and $n(K_{\mathbb P^1}+D^+)$ is integral of degree zero, hence linearly trivial.  Monotonicity and integrality give the rounding condition.
\end{proof}

\begin{lemma}
\label{v-lem:surface-plt-preparation}
Let $(W,B_W)$ be lc but not klt, with $W$ a projective Fano-type surface and $-(K_W+B_W)$ ample.  If $B_W$ has coefficients in a finite set $J$, there is a birational morphism $h\colon Y\to W$ such that
\[
 K_Y+B_Y=h^*(K_W+B_W),\qquad B_Y\geq0,
 \qquad \coeff(B_Y)\subset J\cup\{1\},
\]
and a plt boundary $\Gamma_Y$ with irreducible floor contained in $\floor B_Y$ and $-(K_Y+\Gamma_Y)$ ample.
\end{lemma}
\begin{proof}
If $(W,B_W)$ is dlt, retain one floor component and lower the other coefficients slightly. Otherwise take a nontrivial dlt modification $f:U\to W$ extracting only lc places \cite[Theorem~4.7 and Remark~4.8]{TanakaMMP}. Write $B_U=B_W^\sim+\sum E_i$ and, for small rational $\epsilon>0$, put $D_U=(1-\epsilon)B_W^\sim+\sum E_i$. Since $W$ is klt,
\[
 K_U+D_U=f^*(K_W+(1-\epsilon)B_W)
       +\sum_i a(E_i,W,(1-\epsilon)B_W)E_i
\]
has strictly positive exceptional coefficients. Its MMP over $W$ contracts all these curves by negativity and ends at $W$ \cite[Theorems~1.1 and 4.5]{TanakaMMP}. Retain the last contraction $h:Y\to W$. Its unique exceptional curve $E$ is the floor of a plt boundary $D_Y$, with $-(K_Y+D_Y)$ relatively ample. The crepant boundary $B_Y=B_W^\sim+E$ is effective and lc. For small $\lambda>0$,
\[
 \Gamma_Y=(1-\lambda)B_Y+\lambda D_Y
\]
is plt with floor $E$, and its negative log divisor is ample: it is the sum of $(1-\lambda)h^*(-(K_W+B_W))$ and the relatively ample $\lambda(-(K_Y+D_Y))$. Only $B_Y$ needs the coefficient bound.
\end{proof}

\subsection{Ruled surfaces and multisections}

A horizontal curve can detect the index of a numerically trivial Weil divisor. We give the argument with multiple and reducible fibres included, since both occur after the surface MMP.

\label{v-sec:surface-index}

\begin{lemma}
\label{v-lem:multisection-index}
Let $g:Y\to\mathbb P^1$ be a projective contraction from a normal $\Q$-factorial surface over an algebraically closed field. Suppose its generic fibre is a smooth genus-zero curve. Let $D$ be an integral Weil divisor with $D\equiv0$. If $C$ is a normal horizontal curve, $D$ is Cartier near $C$, and $d=\deg(C/\mathbb P^1)$, then $dD\sim0$.
\end{lemma}
\begin{proof}
The restriction of $D$ to the generic fibre has degree zero. Riemann--Roch on a smooth genus-zero curve, even without a chosen rational point, makes this restriction principal.  Hence there is $\varphi\in k(Y)^*$ such that
\[
 E=D-\Div_Y(\varphi)
\]
is vertical.  It is numerically trivial and is Cartier near $C$. The fibre intersection matrix implies that $E=g^*H$ for a rational divisor $H$ on $\mathbb P^1$. Here is the elementary argument, including reducible and multiple fibres.  Write one fibre as $F=\sum m_iF_i$ and the corresponding part of $E$ as $\sum r_iF_i$.  Set $u=\max_i(r_i/m_i)$. The divisor $uF-E|_{\Supp F}$ is effective, has a zero coefficient, and has zero intersection with every $F_i$.  Positivity of the intersection of distinct meeting curves and connectedness of the fibre force all its coefficients to vanish.  Thus $r_i=um_i$ for every $i$, as asserted.

Put $\gamma=g|_C$.  The Cartier restriction $E|_C=\gamma^*H$ is an integral divisor.  Pushing forward on curves gives
\[
 dH=\gamma_*(E|_C),
\]
so $dH$ is integral.  Also $d\deg H=E\cdot C=0$.  Every integral degree-zero divisor on $\mathbb P^1$ is principal.  Pulling it back gives $dE\sim0$, and therefore $dD\sim0$.  The index $d$ is independent of the multiplicities of the fibres.
\end{proof}

\medskip
We will also use plt preparation over a curve. Quadratic descent is treated together with its generalized version below.

\label{v-sec:surface-dimension-one}

In the next two arguments, the negative log divisor defines a map to a curve and has positive degree on that base. We run the MMP over this curve, so the log canonical divisor remains crepant.

\begin{lemma}
\label{v-lem:relative-plt-preparation}
Let $g:W\to T$ be a projective contraction from a surface of Fano type over a smooth curve.  Suppose $(W,B_W)$ is lc, $C$ is a prime component of $\floor{B_W}$, and $-(K_W+C)$ is $g$-ample. Then there is a birational morphism $h:Y\to W$, either the identity or extracting a single lc place of $(W,B_W)$, and a plt boundary $\Gamma_Y$ with irreducible floor $E$ such that
\[
\begin{gathered}
 K_Y+B_Y=h^*(K_W+B_W),\qquad B_Y\geq0,\\
 E\subset\floor{B_Y},\qquad -(K_Y+\Gamma_Y)\text{ is }gh\text{-ample}.
\end{gathered}
\]
If the coefficients of $B_W$ belong to $J$, those of $B_Y$ belong to $J\cup\{1\}$.  The coefficients of $\Gamma_Y$ are not bounded.
\end{lemma}
\begin{proof}
If $(W,C)$ is plt, take the identity. Otherwise use the last-contraction construction in \cref{v-lem:surface-plt-preparation} for $(W,C)$ over $W$. It gives a single extraction $h:Y\to W$ and a plt boundary $D_Y$ with floor $E$ and $-(K_Y+D_Y)$ $h$-ample. If $K_Y+\Theta_Y=h^*(K_W+C)$, then $\Gamma_Y=(1-\lambda)\Theta_Y+\lambda D_Y$ is plt and has negative log divisor ample over $T$ for small $\lambda>0$. Every lc place of $(W,C)$ is an lc place of $(W,B_W)$, since $C\leq B_W$ and the latter pair is lc. Thus the crepant boundary $B_Y$ is effective with coefficients in $J\cup\{1\}$.
\end{proof}

\subsection{Generalized surface pairs with fixed nef-part index}

In the generalized surface arguments the field is algebraically closed of characteristic $p>5$, $q$ is a positive integer, and $\mathbf M$ is a nef b-divisor with $q\mathbf M$ b-Cartier. Generalized discrepancies are computed with respect to this b-divisor. The input from \cite[Lemmas 3.1, 3.4, 4.4 and 4.5]{YYZ} is generalized adjunction, specified extractions, surface threshold ACC, and surface global ACC. The coefficient sets in these results contain $1/q$: write $\mathbf M=(1/q)(q\mathbf M)$. Their bounds are uniform for $p>5$. We prove the required generalized complement theorem below.

We will repeatedly use two facts about surface nef b-divisors. The trace $M_Y$ on any normal $\Q$-factorial surface model is nef. Indeed, on a higher model $\pi:Y'\to Y$ where $\mathbf M$ descends, the pullback of an effective curve on $Y$ is effective, and the projection formula computes its intersection with $M_Y$ using the nef divisor $M_{Y'}$. Also
\[
 \pi^*M_Y-M_{Y'}\geq0.
\]
This is exceptional and its negative is $\pi$-nef, so negativity applies. Thus generalized lc implies ordinary lc for the boundary on each such model.

\begin{lemma}
\label{gen:ft-extraction}
Suppose $X$ is a Fano-type surface, $(X,D,\mathbf M)$ is generalized lc, and $-(K_X+D+M_X)$ is nef. Let $h:Y\to X$ be a projective birational morphism with $Y$ $\Q$-factorial klt. If its crepant generalized boundary $D_Y$ is effective and has positive coefficient on every exceptional curve, then $Y$ is of Fano type.
\end{lemma}
\begin{proof}
Let $\Gamma$ be a klt log Fano boundary on $X$ and write $K_Y+\Gamma_Y=h^*(K_X+\Gamma)$. The ordinary pair $(Y,D_Y)$ is lc and
\[
 -(K_Y+D_Y)=h^*\bigl(-(K_X+D+M_X)\bigr)+M_Y
\]
is nef. A convex combination of $\Gamma_Y$ and $D_Y$, with weight sufficiently close to one on $D_Y$, is effective klt and has nef and big negative log divisor. Kodaira perturbation proves Fano type.
\end{proof}

\begin{lemma}
\label{gen:curve}
Fix a finite rational seed $I$ and $q>0$. A generalized lc pair $(C,D,M_C)$ with $C=\mathbb P^1$, $D\in\Phi(I)$, $qM_C$ Cartier nef, and $\deg(K_C+D+M_C)\leq0$, has a monotonic $n(I,q)$-complement with unchanged $M_C$ and $q\mid n$.
\end{lemma}
\begin{proof}
Write $e=\deg(qM_C)\in\{0,\ldots,2q\}$. Choose $e$ distinct points outside $\Supp D$, and let $H$ be their sum, so $H\sim qM_C$. Apply \cref{v-lem:curve-complement} to $D+H/q$, with a seed also containing $1-1/q$ and an index divisible by $q$. If its ordinary complement is $D^\dagger$, put $D^+=D^\dagger-H/q$. Then $D^+\geq D$, $D^+\leq1$, and $n(K_C+D^++M_C)\sim0$. Integrality of $nM_C$ and the principal equality give integrality of $nD^+$ and hence the rounding condition.
\end{proof}

\begin{lemma}
\label{gen:rounding}
Fix a rational DCC set $\Phi\subset[0,1]$ and $q>0$. There is a rational $\eta=\eta(\Phi,q)>0$ with the following property. Suppose $X$ is a projective Fano-type surface,
\[
 B=D+\sum c_jE_j,\qquad B^\dagger=D+\sum E_j,
 \quad 1-\eta\leq c_j\leq1,\quad \coeff D\subset\Phi,
\]
where the $E_j$ are distinct and outside $\Supp D$. If $(X,B,\mathbf M)$ is generalized lc and its negative log divisor is nef, then $(X,B^\dagger,\mathbf M)$ is generalized lc and its negative log divisor is pseudo-effective. Its anti-log MMP gives $r:X\to W$ with generalized lc output and nef negative log divisor, and
\[
 r^*(K_W+B_W^\dagger+M_W)\geq K_X+B^\dagger+M_X.
\]
Consequently hyperstandard coefficients reduce to a finite set $J=J(I,q)$, with index-preserving complement transfer.
\end{lemma}
\begin{proof}
Generalized threshold ACC, with boundary set $\Phi$, reduced testing divisor, and nef coefficient $1/q$, supplies a gap below one. This proves generalized lc after rounding.

For any contraction $h:X\to X_1$, put $N=-(K_X+B+M_X)$. Since $N$ is nef, negativity gives $h^*h_*N=N+F$ with $F\geq0$ exceptional. Thus $h_*N$ is nef, by intersection with strict transforms of curves, and
\[
 K_X+B+M_X=h^*(K_{X_1}+h_*B+M_{X_1})+F.
\]
The unrounded generalized pair on $X_1$ is lc. The same threshold argument therefore also applies on every contracted model.

If uniform pseudo-effectivity failed, take $\eta_i\to0$ and run the anti-log MMP to a Mori fibre endpoint $Y_i\to T_i$. On its ray the unrounded log divisor is nonpositive and the rounded one is positive. Interpolation gives $t_i\in[1-\eta_i,1)$ such that
\[
 K_{Y_i}+D_i+t_iE_i+M_{Y_i}\equiv_{T_i}0,
\]
with generalized lc pair and a component of $E_i$ horizontal. Pass to a subsequence with $t_i$ strictly increasing to one. When $T_i$ is a point, generalized surface global ACC contradicts this occurrence of $t_i$. When it is a curve, restriction to a general fibre gives the same contradiction by the degree-two curve calculation: boundary intersection degrees are nonnegative integers and the nef-part degrees are nonnegative multiples of $1/q$. The coefficient set $\Phi\cup\{t_i\}\cup\{1/q\}$ is DCC.

Thus the anti-log MMP ends birationally at a nef model. Reapplying threshold ACC there proves generalized lc; MMP negativity gives the displayed comparison. Pulling back a principal complemented log divisor, with the fixed b-divisor, proves effectivity and preserves its index. For $\Phi(I)$ the coefficients below $1-\eta$ form a finite set. Monotonicity and integrality give the rounding inequality.
\end{proof}

\subsection{Exceptional generalized surfaces}

Call $(X,B,\mathbf M)$ exceptional if every effective rational $P\sim_\Q N:=-(K_X+B+M_X)$ makes $(X,B+P,\mathbf M)$ generalized klt. Such effective representatives exist because $N$ is semiample on a Fano-type surface.

\begin{lemma}
\label{gen:exceptional-gap}
Fix $\Phi$ and $q$ as in \cref{gen:rounding}. There is $\delta=\delta(\Phi,q)>0$ such that every effective $P\sim_\Q N$ for an exceptional pair with $B\in\Phi$ gives a generalized $\delta$-lc pair $(X,B+P,\mathbf M)$.
\end{lemma}
\begin{proof}
Take $0<\delta\leq\eta(\Phi,q)$ with $\delta<1$. Suppose $0<a=a(E,X,B+P,\mathbf M)<\delta$. Extract just $E$, unless it is already on $X$. To extract $E$, raise the other exceptional coefficients on a resolution to numbers strictly between their crepant coefficients and one, then run the relative generalized MMP. Its effective exceptional log divisor is supported exactly on those other curves, so they are contracted and $E$ survives. The crepant boundary $\Omega_Y$ of $(X,B+P,\mathbf M)$ is effective, the generalized pair is klt, and $E$ has coefficient $1-a$. \Cref{gen:ft-extraction} proves Fano type of $Y$.

Write $K_Y+B_Y+M_Y=h^*(K_X+B+M_X)$, let $c=\coeff_E h^*P$, and set $C_Y=B_Y+cE$, $R_Y=h^*P-cE$. Then $C_Y\geq0$ has coefficient $1-a$ on $E$ and the original coefficients elsewhere, and
\[
 -(K_Y+C_Y+M_Y)\sim_\Q R_Y\geq0,\qquad E\not\subset\Supp R_Y.
\]
Run this anti-log MMP to a nef model. The curve $E$ survives because the effective representative excludes it. The generalized log Calabi--Yau completion $C_Y+R_Y$ is crepant throughout; this follows by pulling back its principal rational log divisor. Hence the unrounded output remains generalized klt.

Apply \cref{gen:rounding} to raise the remaining coefficient $1-a$ to one, and choose an effective representative of the nef negative log divisor on the resulting Fano-type surface. Pull back the full log divisor through these comparisons. Its boundary is at least the crepant original boundary and has coefficient at least one on $E$. Pushing to $X$ gives an effective $P^+\sim_\Q N$ for which the generalized pair is not klt. The comparison is crepant, which contradicts exceptionality.
\end{proof}

\begin{theorem}
\label{gen:exceptional}
Fix a finite rational seed $I$ and $q>0$. Every exceptional generalized lc pair $(X,B,\mathbf M)$ on a projective Fano-type surface, with $B\in\Phi(I)$ and $N$ nef, has a monotonic generalized klt $n(I,q)$-complement with the same nef part.
\end{theorem}
\begin{proof}
By \cref{gen:exceptional-gap}, $B$ has coefficients in the finite set $J=\Phi(I)\cap[0,1-\delta]$. Choose $P\geq0$ with $P\sim_\Q N$. The ordinary pair $(X,B+P)$ is $\delta$-lc and its negative log divisor is the nef divisor $M_X$. Mixing it with a klt log Fano boundary gives an ordinary $\epsilon$-klt log del Pezzo pair, for fixed $0<\epsilon<\min\{\delta/2,1/4\}$.

By \cite[Proposition~6.1(d)]{WitaszekBounds}, a uniform integer $Q$ makes every integral Weil divisor on $X$ Cartier after multiplication by $Q$. Choose $n$ divisible by $qQ$ and by $Q$ times the denominators of $J$. Then $nN$ is Cartier. Since $X$ is of Fano type, $\chi(\OO_X)=1$; Riemann--Roch gives $\chi(\OO_X(nN))\geq1$, while Serre duality and bigness of $-K_X$ give $H^2(X,\OO_X(nN))=0$. Choose $G\geq0$ with $G\sim nN$. Exceptionality makes $(X,B+G/n,\mathbf M)$ generalized klt, and it is the required complement. If $N\equiv0$, the same argument gives $G=0$.
\end{proof}

\subsection{Numerically trivial generalized pairs}

\begin{lemma}
\label{gen:index-floor}
Fix a finite rational set $J$ and $q>0$. Suppose $(Y,B,\mathbf M)$ is a projective generalized dlt surface pair, $Y$ is $\Q$-factorial klt, $B\in J$, and $K_Y+B+M_Y\equiv0$. There is $a=a(J,q)$ such that $a(K_Y+B+M_Y)$ is Cartier near $\floor B$, and $q\mid a$. The same assertion holds after removing the floor from $B$.
\end{lemma}
\begin{proof}
For a floor component $C$, generalized adjunction gives
\[
 (K_Y+B+M_Y)|_C=K_C+B_C+M_C,\qquad
 \deg(K_C+B_C+M_C)=0.
\]
Here $C$ is smooth, $B_C\geq0$, $M_C$ is nef and $qM_C$ is Cartier. The coefficient formula in \cite[Lemma 3.1]{YYZ} puts $B_C$ in $\Phi(I(J,q))$: the numerator additions are sums of elements of $J$ and $1/q$, and only finitely many such sums are at most one. Thus $C$ is rational or has genus one. In the rational case, \cref{gen:curve} supplies a bounded index $b=b(J,q)$; the complement increment has degree zero and is effective, so it is zero. Consequently $bB_C$ is integral. In the genus-one case $B_C=0$.

At an intersection of two floor components the underlying surface is smooth, by generalized dlt. At all other points of $C$, the generalized different is strictly less than one. The ordinary different of $(Y,C)$ is at most $B_C$, and its coefficient is $1-1/m_x$, where $m_x$ is the local factorial index \cite[Lemma 3.3]{HW}. Hence $m_x\leq b$ in the rational case and $m_x=1$ in the genus-one case. If $jJ\subset\Z$, the integer
\[
 a=qj\operatorname{lcm}(1,\ldots,b)
\]
makes $K_Y$, every boundary component, and the integral Weil divisor $qM_Y$ Cartier after the required multiplication near the floor. It also works for $B-\floor B$.
\end{proof}

\begin{lemma}
\label{gen:floor-gap}
Fix $J,q$ as above. For generalized dlt log Calabi--Yau pairs on Fano-type surfaces with $B\in J$, the pair $(Y,B-\floor B,\mathbf M)$ is generalized $\epsilon(J,q)$-lc and klt.
\end{lemma}
\begin{proof}
First, all positive generalized discrepancies of $(Y,B,\mathbf M)$ have a uniform gap from zero. Otherwise extract a sequence with discrepancies $a_i\downarrow0$, $a_i<1$. One may extract each specified prime alone: slightly decrease the boundary and nef part to make a generalized klt pair while keeping that discrepancy below one, extract that prime alone, and then restore the original crepant data. The new boundary is the strict transform of $B$ plus $(1-a_i)E_i$, so it is effective and generalized lc, and the log canonical divisor is numerically trivial. Generalized surface global ACC applied to the DCC set $J\cup\{1-a_i\}\cup\{1/q\}$ is a contradiction. If the nef part is numerically trivial, omit it for this ACC use: negativity makes it descend, so it changes neither discrepancies nor numerical triviality.

Removing the floor of a generalized dlt pair makes it generalized klt. Away from that floor, the positive discrepancy gap just proved applies. Near the floor, \cref{gen:index-floor} makes the log divisor Cartier after a uniform multiple $a$ divisible by $q$. On a model where $\mathbf M$ descends, every generalized discrepancy is then in $(1/a)\Z$. Since these discrepancies are positive, they are at least $1/a$. Take the smaller of the two gaps.
\end{proof}

\begin{theorem}
\label{gen:index-zero}
Fix $J,q$. If $(X,B,\mathbf M)$ is generalized lc but not klt, $X$ is a projective Fano-type surface, $B\in J$, and $K_X+B+M_X\equiv0$, then
\[
 n_0(J,q)(K_X+B+M_X)\sim0
\]
for a fixed positive integer $n_0(J,q)$ divisible by $q$.
\end{theorem}
\begin{proof}
Take a generalized dlt modification; it is of Fano type by \cref{gen:ft-extraction}. Write $B=T+\Delta$, with $T=\floor B\ne0$. By \cref{gen:floor-gap}, $(X,\Delta,\mathbf M)$ is generalized $\epsilon(J,q)$-lc and klt. Run its MMP to a Mori fibre space $g:W\to Z$. Generalized discrepancies do not decrease, and every birational step is crepant for the numerically trivial $K+B+M$.

If $Z$ is a point, $(W,\Delta_W)$ is an ordinary $\epsilon$-lc log del Pezzo pair because $M_W$ is nef. The factorial-index bound used in \cref{gen:exceptional}, followed by Riemann--Roch on the Fano-type surface, gives a bounded multiple of $K_W+B_W+M_W$ that is linearly trivial.

If $Z$ is a curve, then $Z=\PP^1$ and the regular generic fibre is a smooth genus-zero curve in characteristic $p>5$. Since the nef part was retained in the MMP,
\[
 T_W\cdot F=-(K_W+\Delta_W+M_W)\cdot F>0.
\]
After a crepant dlt modification, choose a normal horizontal floor component $C$. The degree identity gives $1\leq d=\deg(C/Z)\leq B\cdot F+M\cdot F=2$. By \cref{gen:index-floor}, $D=a(K+B+M)$ is integral, numerically trivial and Cartier near $C$ for bounded $a=a(J,q)$. Hence \cref{v-lem:multisection-index} gives $dD\sim0$. A common multiple of $2a$ and the point-base index, transferred crepantly to the original surface, proves the assertion.
\end{proof}

\subsection{Lifting from a generalized plt curve}

\begin{lemma}
\label{gen:inversion-connectedness}
On a projective $\Q$-factorial klt surface, the following hold. First, the generalized non-klt locus is connected if the negative generalized log divisor is ample. Second, let $C$ be an irreducible coefficient-one curve. Suppose a generalized plt auxiliary pair with floor $C$ exists. If the generalized adjoint pair of a given boundary on $C$ is lc, then the given pair is generalized lc near $C$. The nef parts of the given pair and auxiliary pair may differ.
\end{lemma}
\begin{proof}
We prove both assertions by approximating the nef part with an effective divisor on a resolution. Let $\phi:Y\to X$ resolve all the boundaries and the nef parts, and write $K_Y+D_Y+M_Y=\phi^*(K_X+D+M_X)$. Choose an ample $A$ on $X$ and an effective exceptional $E$ with $-E$ relatively ample. For sufficiently small positive rational $\epsilon,\tau$,
\[
 M_Y+\epsilon\phi^*A-\tau E
\]
is ample. Choose a sufficiently diluted general effective member $H_Y$ of this rational linear system. It can be chosen smooth with snc union with the fixed resolution divisor and with arbitrarily small coefficients. The ordinary boundary $D^{\rm ord}=D+\phi_*H_Y$ has crepant transform
\[
 D_Y^{\rm ord}=D_Y+\tau E+H_Y.
\]
Choose $\tau$ below all the finitely many positive coefficient gaps to one. Then the images of the non-klt strata of these two transforms are identical. Thus the ordinary non-klt locus equals the generalized one. If the generalized negative log divisor is ample, choose $\epsilon$ small enough that its difference with $\epsilon A$ is ample. Ordinary surface connectedness \cite[Theorem 5.2]{TanakaMMP} proves the first assertion.

For the second assertion, suppose the given pair has a negative discrepancy over $C$. Interpolate it with the auxiliary generalized plt pair, giving the original pair weight less than, but sufficiently close to, one. The specified discrepancy stays negative, whereas the generalized different on $C$ is klt. On the fixed resolution the strict transform $C'$ is isomorphic to $C$. In the preceding approximation choose $\tau$ so that adding $\tau E|_{C'}$ keeps the different below one; choose $H_Y$ to avoid its support on $C'$ and to have restriction coefficients less than one. Then ordinary adjunction on $C$ is klt. Ordinary plt inversion \cite[Theorem 5.1]{TanakaMMP} makes the ordinary pair plt near $C$. This contradicts the negative discrepancy retained by the approximation.
\end{proof}

\begin{proposition}
\label{gen:plt-lifting}\label{v-prop:surface-plt-lifting}
Fix a finite rational set $J$ and $q>0$. Suppose $X$ is a projective Fano-type surface, $(X,B,\mathbf M)$ is generalized lc, $B\in J$, and $N=-(K_X+B+M_X)$ is nef. Suppose an auxiliary generalized plt pair $(X,\Gamma,\mathbf P)$ has irreducible floor $C\subset\floor B$ and $A=-(K_X+\Gamma+P_X)$ ample, with $\mathbf P$ nef and rational b-Cartier. Then the pair has a monotonic $n(J,q)$-complement with unchanged $\mathbf M$. The index depends only on $J$ and $q$. For ordinary pairs the stated ordinary plt boundary already implies Fano type, by lowering its floor.
\end{proposition}
\begin{proof}
Generalized adjunction from the auxiliary pair makes $C=\mathbb P^1$. Adjunction for the given pair gives $B_C\in\Phi(I(J,q))$, $qM_C$ Cartier nef, and $\deg(K_C+B_C+M_C)\leq0$. By \cref{gen:curve}, choose $B_C^+$ with index $n$ divisible also by the denominators of $J$.

Resolve both b-divisors and the two boundaries by $\phi:Y\to X$. Use canonical representatives with coefficient $-1$ along $C,C'$. Write
\[
 K_Y+B'+M_Y=\phi^*(K_X+B+M_X),\quad
 T'=(B')^{=1},\quad \Delta'=B'-T',\quad N'=\phi^*N,
\]
and put
\[
 Q'=n\Delta'-\floor{(n+1)\Delta'},\qquad
 L'=-nK_Y-nM_Y-nT'-\floor{(n+1)\Delta'}=nN'+Q'.
\]
The divisor $L'$ is integral because $q\mid n$. Interpolate the auxiliary and given generalized pairs, writing their data as $(\widehat\Gamma,\widehat{\mathbf P})$, with weight on the given pair close to one but less than one. It is generalized plt with floor $C$ and has ample negative log divisor $\widehat A$. Its crepant boundary $\widehat\Gamma'$ converges coefficientwise to $B'$. For $E\ne C'$ set
\[
 p_E=-\floor{\coeff_E(\widehat\Gamma'+Q')},\qquad
 P'=\sum p_EE,\qquad \Lambda'=\widehat\Gamma'+Q'+P'.
\]
The divisor $P'$ is effective, reduced and exceptional, and $\Lambda'$ is an snc boundary with floor $C'$. Indeed, at a coefficient-one component of $B'$ other than $C'$ the coefficient of $\widehat\Gamma'$ is in $(0,1)$; at every other prime the expression is $\widehat\gamma'_E-b'_E+\{(n+1)b'_E\}\in(-1,1)$. At nonexceptional primes $Q'=0$ because $nB$ is integral.

Set $D'=L'+P'-C'$ and $D=nN-C$. Then
\[
 D'=K_Y+(\Lambda'-C')+\widehat P_Y+
                 \phi^*(\widehat A+nN).
\]
Birational surface vanishing \cite[Lemma 2.23]{HX} gives $R^1\phi_*\OO_Y(D')=0$, since $\widehat P_Y$ is $\phi$-nef. We compute the zeroth direct image directly. One has
\[
 D'-\phi^*D=Q'+P'+\phi^*C-C'.
\]
On an exceptional curve, $Q'+P'=\Lambda'-\widehat\Gamma'>-1$, and $\phi^*C-C'$ is effective. If $s$ is a local rational section of $\OO_X(D)$, then $\Div(s)+D'$ is an integral divisor whose exceptional coefficients are greater than $-1$ and whose nonexceptional coefficients are nonnegative. It is therefore effective. The reverse inclusion follows by pushforward. Hence $\phi_*\OO_Y(D')=\OO_X(D)$.

The ordinary pair $(X,\widehat\Gamma-C)$ is klt and
\[
 D-(K_X+\widehat\Gamma-C)=nN+\widehat A+\widehat P_X
\]
is ample. The trace $\widehat P_X$ is nef on the surface. The integral Weil form of \cite[Theorem 1.1]{ABL} and Leray give $H^1(Y,\OO_Y(D'))=0$.

Let $R_C=B_C^+-B_C$, and choose $h\in k(C)^*$ with $\Div(h)=n(K_C+B_C^++M_C)$. The divisor
\[
 G_{C'}=(L'+P')|_{C'}+\Div(h)
        =nR_C+(Q'+P')|_{C'}
\]
is integral and effective. A negative rounding correction is greater than $-1$, and at a point of $C'$ at most one other snc component contributes. Thus its integral coefficients are nonnegative. The Cartier restriction sequence lifts it to $G'\geq0$ with $G'\sim L'+P'$ and the prescribed restriction. Put $G=\phi_*G'$ and $B^+=B+G/n$. Then $n(K_X+B^++M_X)\sim0$. The exceptional numerically trivial difference
\[
 G'-P'-Q'-\phi^*G
\]
is zero by negativity. Restriction therefore recovers exactly $K_C+B_C^++M_C$. \Cref{gen:inversion-connectedness} gives generalized lc near $C$. If it failed elsewhere, an interpolation with the auxiliary plt pair would have ample negative log divisor and a disconnected generalized non-klt locus, again contradicting that lemma. Integrality supplies the complement rounding condition.
\end{proof}

\begin{proposition}
\label{gen:nonklt-big}
For fixed $J,q$, a generalized lc but not klt pair on a projective Fano-type surface, with $B\in J$ and $N=-(K+B+M)$ nef and big, has a monotonic $n(J,q)$-complement.
\end{proposition}
\begin{proof}
Take a generalized dlt model, justified by \cref{gen:ft-extraction}, and choose $C\subset\floor B$. Let $g:X\to V$ be the semiample contraction of $N$; it is birational. Put $D=-K_X-C$. Over $V$,
\[
 D\equiv_V B-C+M_X.
\]
The trace $M_X$ is nef and relatively semiample on the Fano-type surface. Thus the right side has an effective representative over $V$ excluding $C$. Run a $D$-MMP over $V$ to a nef model. The curve $C$ survives: if it were a contracted negative curve, $(B-C)\cdot C+M_X\cdot C\geq0$ would contradict negativity. All steps are crepant for the generalized log divisor, which is pulled back from $V$; the boundary coefficients remain in $J\cup\{1\}$ and $N$ stays nef and big.

Choose $m$ sufficiently large that $mN+D$ is nef and is strictly positive on every extremal ray where $N$ is positive. This is possible by rational polyhedrality. The ordinary lc boundary
\[
 \Delta=\frac{mB+C}{m+1}
\]
then has
\[
 -(K_X+\Delta)=\frac{mN+D+mM_X}{m+1}
\]
nef and big. Its ample contraction contracts only $N$-trivial curves, so it is also crepant for the generalized log divisor. Its ordinary lc log-Fano output is not klt, since $C$ was an lc place. Apply the complete ordinary plt preparation \cref{v-lem:surface-plt-preparation} to that output. Every extracted lc place of the ordinary pair $(X,\Delta)$ is a generalized lc place of $(X,B,\mathbf M)$: $\Delta\leq B$ and nef-part negativity compares the discrepancies. Thus the crepant generalized pair on the plt model has coefficients in $J\cup\{1\}$ and nef negative log divisor. Its ordinary plt auxiliary boundary has zero nef part. Apply \cref{gen:plt-lifting} and transfer back.
\end{proof}

\subsection{Generalized pairs fibred over a curve}

\begin{proposition}
\label{gen:dimension-one-plt}
Fix $J,q$. Let $(X,B,\mathbf M)$ be a projective generalized dlt pair with $X$ of Fano type, $B\in J$, and $N=-(K+B+M)$ nef of numerical dimension one. Let $g:X\to T$ be its semiample contraction. If a component $C$ of $\floor B$ is vertical or has degree one over $T$, there is a monotonic $n(J,q)$-complement.
\end{proposition}
\begin{proof}
Write $N\sim_\Q g^*A$ with $A$ ample; $T=\mathbb P^1$ since $H^1(X,\OO_X)=0$. Every birational contraction over $T$ is crepant for the generalized log divisor, by negativity, so it preserves generalized lc, the nef b-divisor, and the finite boundary coefficients.

If $C$ is vertical, write $g^*z=mC+F'$ with $F'$ excluding $C$. Run a $(-C)$-MMP over $T$, using $-C\sim_{\Q,T}F'/m$ to prove survival of $C$. At its nef endpoint, $C$ cannot meet another component of its fibre. Thus it is the whole reduced fibre. The relative $K$-MMP then gives a Mori fibre surface $W\to T$, preserving $C$, with $-(K_W+C)$ relatively ample. If $C$ has degree one, run the MMP and relative ample contraction of $-K_X-C$, whose generic-fibre degree is one. They are birational over $T$ and preserve the horizontal $C$, giving the same positivity.

The ordinary pair $(W,C)$ is lc. The relative ordinary plt preparation \cref{v-lem:relative-plt-preparation} gives either the identity or a single extraction $h:Y\to W$, and an ordinary plt boundary $\Gamma_Y$ with irreducible floor $E$ and relatively ample negative log divisor. Any extracted lc place of $(W,C)$ is a generalized lc place of the given pair, so its crepant boundary $B_Y$ has coefficients in $J\cup\{1\}$ and $E\subset\floor B_Y$. \Cref{gen:ft-extraction} proves Fano type.

For small $t>0$, the generalized auxiliary data
\[
 \Gamma_t=(1-t)B_Y+t\Gamma_Y,\qquad
 \mathbf P_t=(1-t)\mathbf M
\]
are plt with floor $E$. Their negative log divisor is
\[
 (1-t)(gh)^*A+t\bigl(-(K_Y+\Gamma_Y)\bigr),
\]
which is ample globally. Apply \cref{gen:plt-lifting} and reverse the crepant comparisons. The parameter $t$ does not enter the index.
\end{proof}

\begin{proposition}
\label{gen:dimension-one-bisection}\label{v-prop:dimension-one-bisection}
In the preceding setting, if a component $C$ of the floor has degree two over $T$, there is a monotonic $n(J,q)$-complement.
\end{proposition}
\begin{proof}
On the generic fibre, $\deg B+\deg M=2$. Hence $C$ is the entire horizontal boundary and $M$ has generic-fibre degree zero. Adjunction gives
\[
 (K_X+B+M_X)|_C=K_C+B_C+M_C,\qquad
 \deg(K_C+B_C+M_C)=-2\deg A<0.
\]
Thus $C=\mathbb P^1$, $B_C\in\Phi(I(J,q))$, and $qM_C$ is Cartier nef. The double cover $\gamma:C\to T$ is separable and Galois because $p>5$.

Normalize the base change to obtain $\pi:U\to X$ with involution $\iota$. Pull back the nef b-divisor on a high model and define the crepant generalized subboundary by
\[
 K_U+B_U+M_U=\pi^*(K_X+B+M_X).
\]
The boundary $B_U$ may have negative vertical coefficients. The tame-cover discrepancy formula \cite[Proposition~99 and Corollary~100]{KollarCovers} multiplies log discrepancies by the ramification index. Apply it on a model where the nef part descends; the same identity holds on every higher model. Thus the generalized subpair is lc.

Over $k(C)$ the bisection splits into two rational points. Use the ruled model $r:Z=\mathbb P^1\times C\to C$ with sections $\Sigma_0,\Sigma_\infty$ at these points. We verify that the nef b-divisor descends to this model. On $Z$, the integral divisor $qM_Z$ has degree zero on the generic fibre, so it is principal there. After subtracting a $(1/q)$-principal b-divisor, its trace is vertical. On a common resolution $Y\to Z$, the nef divisor $M_Y$ is then vertical too. A vertical divisor that is nef over a curve has zero degree on every fibre component: its intersection with the whole fibre is zero and all the summands are nonnegative. The fibre-matrix argument in \cref{v-lem:multisection-index} makes it a pullback from $C$. Consequently
\[
 M_Z=r^*D_M,\qquad M_Y=r_Y^*D_M,
\]
which proves descent as a b-divisor. Since $Z$ is smooth, $qM_Z$ and $qD_M$ are Cartier. The harmless change of representative is undone by a principal divisor for every index divisible by $q$.

Transport the log canonical divisor crepantly to $Z$. Its boundary is
\[
 B_Z=\Sigma_0+\Sigma_\infty+r^*D_B.
\]
Restriction to either marked section gives $D_B=B_C$ and $D_M=M_C$ with compatible representatives. Indeed, the section maps isomorphically to the original normal $C$, with ramification index one at its generic point. A rational canonical form with a simple pole along $C$ therefore has the same residue after pullback. The residue definition of the different \cite[Paragraph~13]{KollarCovers} identifies $D_B$ with $B_C$. Restriction of the Cartier line bundle $q\mathbf M$ identifies $D_M$ with $M_C$.

Take a complement $B_C^+$ from \cref{gen:curve}, with index $a$ divisible by $q$. The generalized pair on $Z$ with boundary $\Sigma_0+\Sigma_\infty+r^*B_C^+$ and nef part $r^*M_C$ is lc, and its log canonical divisor has principal $a$-fold multiple. Transport it to $U$: the increment is $R_U=g_U^*(B_C^+-B_C)\geq0$. Set
\[
 B^+=B+\tfrac12\pi_*R_U.
\]
Its crepant generalized pullback is the average of this lc subpair and its involute, hence is lc. Tame discrepancy descent gives generalized lc downstairs. Finally, the norm of a rational trivialization upstairs gives
\[
 2a(K_X+B^++M_X)\sim0.
\]
The index $2a$ is divisible by $q$, and monotonicity and integrality give the complement rounding.
\end{proof}

\subsection{Global surface complements}

\begin{theorem}[Generalized surface complements]
\label{gen:global-surface}
Fix a finite rational seed $I$ and $q>0$. There is an integer $n=n(I,q)>0$ divisible by $q$ such that every generalized lc pair $(X,B,\mathbf M)$ on a projective Fano-type surface, with $B\in\Phi(I)$, $q\mathbf M$ b-Cartier nef, and $N=-(K_X+B+M_X)$ nef, has a monotonic generalized $n$-complement with unchanged nef b-divisor. The integer is uniform for $p>5$.
\end{theorem}
\begin{proof}
The exceptional case is \cref{gen:exceptional}. Otherwise choose $P\geq0$, $P\sim_\Q N$, making the pair non-klt. Put $\Omega=B+tP$, where $t$ is its generalized lc threshold if the original pair is klt, and $t=0$ otherwise. On a generalized dlt modification, let $\Theta$ be the strict transform of $B$, with coefficient one on every exceptional curve and on a chosen component of $\floor\Omega_Y$. Then
\[
 B_Y^{\rm crep}\leq\Theta\leq\Omega_Y,\qquad
 \coeff\Theta\subset\Phi(I)\cup\{1\}.
\]
The model is of Fano type by \cref{gen:ft-extraction}. Complete the semiample $-(K_Y+\Omega_Y+M_Y)$ by a sufficiently diluted effective divisor preserving generalized lc. On a resolution of the nef data, this is ordinary semiample perturbation \cite[Theorem~1]{TanakaPerturb}. It gives an lc log Calabi--Yau completion of $\Theta$ whose increment excludes the chosen floor. Thus its anti-log MMP ends at a nef model and preserves that floor; the completion remains crepant and the output generalized lc.

Apply \cref{gen:rounding} to obtain a finite set $J=J(I,q)$. The chosen lc place remains lc, since its discrepancy can only decrease and the rounded pair is lc. Display it again on a generalized dlt model, which is of Fano type by \cref{gen:ft-extraction}.

For the resulting non-klt finite-coefficient pair, numerical dimension zero is \cref{gen:index-zero}, and dimension two is \cref{gen:nonklt-big}. In dimension one the generic fibre has genus zero and total boundary plus nef-part degree two. A nonempty floor therefore has either a vertical component, a section or a bisection. Apply \cref{gen:dimension-one-plt,gen:dimension-one-bisection}. A common multiple of these indices gives the required $n(I,q)$. Crepant or sub-crepant transfer with the fixed nef b-divisor preserves effectivity and the principal multiple; monotonicity and integrality give the original complement inequality.
\end{proof}

\begin{corollary}
\label{v-thm:global-surface-complements}
\label{v-cor:final-global-remainder}
For every finite rational set $I\subset[0,1]$ there is an integer $n(I)$ such that every projective lc surface pair $(V,B)$ over an algebraically closed field of characteristic $p>5$, with $V$ of Fano type, coefficients in $\Phi(I)$, and $-(K_V+B)$ nef, has a monotonic $n(I)$-complement.
\end{corollary}
\begin{proof}
Apply \cref{gen:global-surface} with $q=1$ and $\mathbf M=0$.
\end{proof}

\subsection{Surface complements over a curve}

\label{v-sec:relative-surface}

We now work over an affine neighbourhood of a specified closed point $z$ on a smooth curve $T$.  The surface is projective over $T$; neither projectivity over $k$ nor rationality of $T$ is required.

\begin{lemma}
\label{v-lem:relative-surface-tools}
Let $g:V\to T$ be a projective contraction, where $V$ is of Fano type over $T$.  The following operations are valid after shrinking $T$.
\begin{enumerate}[label=(\roman*)]
\item An MMP over $T$ for any rational divisor exists and terminates. Its birational targets are $\Q$-factorial and of Fano type over $T$. A nef rational divisor is semiample over $T$.
\item If $(V,B)$ is lc and $-(K_V+B)$ is nef over $T$, every crepant dlt modification extracting only lc places is of Fano type over $T$. One may arrange that any existing vertical lc centre over $z$ gives a vertical component of the floor on this modification.
\end{enumerate}
\end{lemma}
\begin{proof}
Choose a klt boundary $\Gamma$ with $-(K_V+\Gamma)$ ample over $T$. For any rational divisor $D$ and small $\epsilon>0$, $\epsilon D-(K_V+\Gamma)$ is relatively ample. Semiample perturbation \cite[Theorem~1]{TanakaPerturb} therefore realizes a $D$-MMP as a klt surface MMP, which exists and terminates by \cite[Theorems~1.1 and 4.5]{TanakaMMP}. Under a birational contraction $h$, negativity gives $h^*h_*(-(K_V+\Gamma))=-(K_V+\Gamma)+F$ with $F\geq0$ exceptional. Thus the pushforward is nef and big over $T$, and the pushed pair is klt. Kodaira perturbation preserves relative Fano type. The targets are $\Q$-factorial by the surface contraction theorem. Semiampleness of a nef divisor follows from \cite[Theorem~4.2]{TanakaMMP}; its point-base exception does not occur here.

For (ii), semiampleness supplies an effective lc completion $\Omega_Y\geq B_Y$ with $K_Y+\Omega_Y\sim_{\Q,T}0$ on a crepant dlt modification. Interpolating it with the crepant transform of $\Gamma$ gives an effective klt boundary with nef and big negative log divisor, since every exceptional coefficient of $\Omega_Y$ is one. This proves relative Fano type. If the desired vertical lc place is not a floor component, its centre is a floor intersection; blowing up that point displays a vertical coefficient-one curve, to which the same argument applies.
\end{proof}

\begin{lemma}
\label{v-lem:relative-rounding}
The conclusion of \cref{v-rem:finite-reduction} holds for pairs of Fano type over $T$ with nef negative log divisor over $T$. The new finite coefficient set and the complement-transfer index depend only on the original hyperstandard set.
\end{lemma}
\begin{proof}
Apply the proof of \cref{gen:rounding} with nef part zero over $T$, using \cref{v-lem:relative-surface-tools}. Threshold ACC is local, and negativity preserves the unrounded nef negative log divisor on every birational model. A non-pseudo-effective rounded divisor would end in a Mori contraction over $T$; its generic fibre is a curve. Interpolation then contradicts relative surface global ACC \cite[Proposition~11.7]{BirkarCharP}, exactly as in the curve-base case of that proof. The same negativity comparison gives effective complement transfer. Pullback of its principal multiple near $z$ preserves the index.
\end{proof}

\begin{proposition}
\label{v-prop:relative-plt-lifting}
Fix a finite rational set $J$.  Let $(V,B)$ be lc with coefficients in $J$ and $N=-(K_V+B)$ nef over $T$. Suppose $(V,\Gamma)$ is $\Q$-factorial plt, its irreducible floor $C$ is vertical over $z$ and a component of $\floor B$, and $A=-(K_V+\Gamma)$ is ample over $T$. Then $(V,B)$ has a monotonic $n(J)$-complement over $z$. The index does not depend on the coefficients of $\Gamma$.
\end{proposition}
\begin{proof}
Adjunction makes $C\simeq\PP^1$, with lc different $B_C\in\Phi(I(J))$ and $\deg(K_C+B_C)\leq0$. Choose its complement by \cref{v-lem:curve-complement}, making $n$ clear $J$ as well. For the klt pair $(V,\Gamma-C)$, the integral Weil divisor $nN-C$ satisfies
\[
 (nN-C)-(K_V+\Gamma-C)=nN+A.
\]
Here $nN-C$ is integral and $\Q$-Cartier, the pair $(V,\Gamma-C)$ is klt so all its lc centres dominate $T$, and $nN+A$ is relatively ample; also $2$ is invertible. Thus \cite[Theorem~2.51]{ABP} gives $R^1g_*\OO_V(nN-C)=0$. Since $T$ is affine, its first cohomology vanishes.

Apply the resolution construction in \cref{gen:plt-lifting} with nef part zero. Its exceptional correction, zeroth direct image and birational surface vanishing are local over $V$. The vanishing just proved consequently lifts the prescribed curve section and gives $G\geq0$ with $G\sim nN$ and
\[
 B^+=B+G/n,\qquad (K_V+B^+)|_C=K_C+B_C^+.
\]
Inversion gives lc near $C$. If lc failed elsewhere on $V_z$, an interpolation with $\Gamma$ would have ample negative log divisor and a disconnected non-klt locus in that fibre, contrary to \cite[Theorem~5.2]{TanakaMMP}. Shrinking $T$ gives the required complement.
\end{proof}

\begin{proposition}
\label{v-prop:relative-vertical-reduction}
Let $(V,B)$ be lc and of Fano type over $T$, with coefficients in a finite rational set $J$ and nef negative log divisor. There is a relatively Fano-type lc pair $(W,\Theta_W)$, with coefficients in $J\cup\{1\}$, nef negative log divisor and a vertical floor component over $z$, such that every monotonic $n$-complement on $W$ transfers to $(V,B)$ with the same index.
\end{proposition}
\begin{proof}
Shrink $T$ so that $F=g^*z$ is principal, and put
\[
 t=\operatorname{lct}_{V_z}(V,B;F),\qquad \Omega=B+tF.
\]
The threshold is finite, nonnegative and rational, as is seen on a log resolution of $B+F$.  At the threshold there is a vertical lc place over $z$.  This includes $t=0$, when such a place already exists.  The negative log divisor of $\Omega$ is rationally linearly equivalent to $N=-(K_V+B)$, so it is semiample over the affine base.

Take a crepant dlt modification $h:Y\to(V,\Omega)$ with a vertical floor component over $z$, as in \cref{v-lem:relative-surface-tools}. Define $\Theta_Y$ as the strict transform of $B$, with all components of $\floor\Omega_Y$ raised to coefficient one; every exceptional curve already has that coefficient.  Then
\[
 0\leq\Theta_Y\leq\Omega_Y,\quad
 \coeff(\Theta_Y)\subset J\cup\{1\},\quad
 \Theta_Y\geq B_Y,
 \qquad K_Y+B_Y=h^*(K_V+B).
\]
For the last inequality, the exceptional coefficient of $B_Y$ is at most one, by lc of $(V,B)$. Set $R_Y=\Omega_Y-\Theta_Y$; it is effective and excludes the floor. Semiample perturbation supplies $Q_Y\geq0$ such that $\Omega_Y+Q_Y$ is lc and $K_Y+\Omega_Y+Q_Y\sim_\Q0$. It excludes the floor as well.  Fano type of $Y$ follows from \cref{v-lem:relative-surface-tools}.

Run the $D_Y=-(K_Y+\Theta_Y)$-MMP over $T$. Since $D_Y\sim_\Q Q_Y+R_Y\geq0$, it ends with a nef model $r:Y\to W$.  This effective representative contains no floor component, so the vertical floor survives all negative contractions. The lc completion stays crepant, because its log divisor is rationally principal.  Its pushforward dominates $\Theta_W$; hence $(W,\Theta_W)$ is lc, and Fano type descends. Finally
\[
 r^*(K_W+\Theta_W)\geq K_Y+\Theta_Y\geq h^*(K_V+B).
\]
Pulling an $n$-complement to $Y$ and pushing it to $V$ therefore gives an effective lc boundary above $B$.  Pullback and pushforward of its principal $n$-fold log divisor preserve the index. The complement index is independent of $t$ and of the multiplicities of $F$.
\end{proof}

\begin{theorem}[Relative surface complements]
\label{v-thm:relative-surface-complements}
Fix a finite rational set $I\subset[0,1]$. There is an integer $n(I)$ such that, for every projective contraction $g:V\to W$ from an lc surface pair $(V,\Delta)$ to a normal curve in characteristic $p>5$, with $V$ of Fano type over $W$, $\coeff\Delta\subset\Phi(I)$, and $-(K_V+\Delta)$ nef over $W$, a monotonic $n(I)$-complement exists over each point of $W$.
\end{theorem}
\begin{proof}
By \cref{v-lem:relative-rounding,v-prop:relative-vertical-reduction}, we may assume finite coefficients, a vertical floor component $C$ over $z$, and $N=-(K_V+B)$ nef over $T$. A crepant dlt modification preserves these properties and relative Fano type. The regular generic fibre is a smooth genus-zero curve in characteristic $p>5$.

If $N$ has positive generic-fibre degree, its semiample contraction $s:V\to W$ is birational and $N=s^*N_W$ with $N_W$ relatively ample. The log pair on $W$ has a vertical lc place over $z$, given by $C$. If this is a floor component of a dlt pair, retain it and lower all other boundary coefficients slightly. Otherwise choose a crepant dlt modification displaying a vertical floor, with all exceptional curves over $z$ after shrinking $T$. Lower the strict transform of $B_W$ and run its MMP over $W$, as in \cref{v-lem:surface-plt-preparation}. The last contraction has one vertical exceptional curve and a plt boundary relatively ample over $W$. Mixing this boundary with the crepant $B_W$ gives a plt boundary whose negative log divisor is ample over $T$. In both cases apply \cref{v-prop:relative-plt-lifting} and transfer crepantly to $V$.

If $N$ has zero generic-fibre degree, semiampleness gives $N\sim_{\Q,T}0$. Every birational contraction over $T$ is consequently crepant for the log divisor. Write $g^*z=mC+F'$ with $F'$ excluding $C$. Run a $(-C)$-MMP using $-C\sim_{\Q,T}F'/m$; the effective representative ensures survival of $C$. At the nef endpoint, intersection and connectedness make $C$ the whole reduced fibre. The relative $K$-MMP preserves it, since $C^2=0$. On the Mori model, $-(K_W+C)$ is relatively ample. Apply \cref{v-lem:relative-plt-preparation,v-prop:relative-plt-lifting}.

All coefficient sets depend only on $I$. A common multiple of the resulting curve-complement indices gives $n(I)$, and crepant transfer preserves this index. The generic-point assertion follows from the closed-point assertion.
\end{proof}

\section{Relative threefold complements}
\label{sec:relative-complements}

We now lift the surface complements to a threefold near a prescribed fibre. The main issue is vanishing for an integral Weil divisor. An ordered MMP reduces this vanishing to a model on which the negative of the chosen centre is nef. We then lift a prescribed surface complement and prepare a centre to which that lifting applies.

\subsection{A plt centre and its auxiliary boundary}\label{v-sec:target}

Consider an lc threefold pair $(X,B)$ and a contraction $f\colon X\to Z$ in characteristic $p>5$, with the following assumptions near a prescribed point $z\in Z$:
\begin{enumerate}[label=(\arabic*)]
\item $0<\dim Z<3$ and the coefficients of $B$ belong to a fixed finite rational set $R$;
\item $X$ is of Fano type over $Z$;
\item there is a rational boundary $\Gamma$ such that $(X,\Gamma)$ is $\Q$-factorial plt;
\item $S=\floor\Gamma$ is irreducible, vertical over $Z$, is a component of $\floor B$, and meets $X_z$;
\item $A=-(K_X+\Gamma)$ is ample over $Z$, and $N=-(K_X+B)$ is nef over $Z$;
\end{enumerate}
We refer to these assumptions as the plt hypotheses.

\medskip

\label{v-sec:preparation}

We shrink to an affine neighbourhood of $z$ in the base.

\begin{lemma}
\label{v-lem:omega}
There is an effective rational boundary $\Omega\geq\Gamma$ such that
\begin{equation}
 (X,\Omega)\text{ is plt},\qquad
 \floor\Omega=S,\qquad K_X+\Omega\sim_{\Q,Z}0.
 \label{v-eq:omega}
\end{equation}
Its coefficients need not belong to a fixed set.
\end{lemma}
\begin{proof}
Take a log resolution $\rho:U\to(X,\Gamma)$ and write $K_U+\Gamma_U=\rho^*(K_X+\Gamma)$. Choose an effective exceptional $E$ with $-E$ relatively ample. For small rational $t>0$, $\rho^*A-tE$ is ample over $Z$ and the exceptional coefficients of $\Gamma_U+tE$ are below one. A sufficiently diluted general member $H/m$ of this ample class makes $(U,\Gamma_U+tE+H/m)$ sub-plt, with only coefficient-one divisor the transform of $S$. Set $\Omega=\Gamma+\rho_*H/m$. The exceptional difference between $tE+H/m$ and $\rho^*(\rho_*H/m)$ is numerically trivial over $X$, hence zero. Thus this is the crepant pullback of $(X,\Omega)$ and $K_X+\Omega\sim_{\Q,Z}0$.
\end{proof}

\begin{lemma}
\label{v-lem:survival}
The $-S$-MMPs below are pseudo-effective klt MMPs.  The centre $S$ survives.  The auxiliary pair remains crepant plt on every MMP model and on every contraction target.  Every $\Q$-factorial MMP model remains of Fano type over $Z$.
\end{lemma}
\begin{proof}
Choose an effective Cartier divisor $J$ on $Z$ containing $f(S)$ and write
\begin{equation}
 f^*J=mS+F,\qquad F\geq0,\quad S\not\subset\Supp F,
 \qquad -S\sim_{\Q,Z}F/m.
 \label{v-eq:pseudoeffective}
\end{equation}
For $0<\epsilon<1$,
\begin{equation}
 K_X+\Omega-\epsilon S\sim_{\Q,Z}-\epsilon S
 \label{v-eq:kltmmp}
\end{equation}
is pseudo-effective and klt. Its MMP terminates by \cite[Theorems~1.1--1.2]{YYZ}, with zero nef part, and cannot end in a Mori fibre space. The divisor $S$ survives: a general contracted curve covering it has nonnegative intersection with the effective representative $F/m$. The same argument applies on every model.

The auxiliary log divisor is pulled back from $Z$ up to a rational principal divisor, so every step is crepant for it. Since $S$ survives, the resulting effective pair remains plt. To see preservation of Fano type, apply the construction of \cref{v-lem:omega} to a klt log Fano boundary. It gives a big klt log Calabi--Yau boundary whose pushforward remains effective, big and klt. Kodaira perturbation produces a klt log Fano boundary on each $\Q$-factorial model.
\end{proof}

\subsection{Birational transport and restriction}

\label{v-sec:transport}

\begin{lemma}
\label{v-lem:transport}
Let $h\colon V\to W$ be a projective birational morphism of normal threefolds over $k$, with effective plt pairs satisfying
\[
 K_V+\Omega_V=h^*(K_W+\Omega_W).
\]
Their coefficient-one divisors are $S_V$ and $S_W$, and $S_V$ is not contracted.  Let $H_W$ be $\Q$-Cartier and $H_V=h^*H_W$.  If $L_V=H_V-aS_V$ is integral Weil and $0\leq a\leq1$, then, for $L_W=h_*L_V$,
\begin{equation}
 h_*\OO_V(L_V)=\OO_W(L_W),\qquad
 R^ih_*\OO_V(L_V)=0\quad(i>0).
 \label{v-eq:transport}
\end{equation}
\end{lemma}
\begin{proof}
Use the Fujita transforms of $(L_W,aS_W)$, with ambient boundary $\Omega_W$, as defined in \cite[Definition~8]{BK}.  These are legitimate because $L_W+aS_W=H_W$ is $\Q$-Cartier and $0\leq aS_W\leq\Omega_W$.

The ordinary Fujita transform is exactly $L_V$, with associated boundary $aS_V$: their sum is pulled back, their pushforwards are prescribed, and the exceptional boundary coefficients are zero.  Ordinary Fujita pushforward therefore gives the first assertion of \eqref{v-eq:transport}. It is also the $K$-twisted Fujita transform, with complementary boundary $\Omega_V-aS_V$.  Indeed
\[
 L_V\sim_{\Q,h}K_V+\Omega_V-aS_V,
\]
and this complementary boundary is effective, with exceptional coefficients in $[0,1)$.

Choose a log resolution $\mu\colon U\to V$ such that $h\mu$ is a log resolution over $W$ as well.  Let $L_U$ be the $K$-twisted Fujita transform of $(L_V,aS_V)$ with ambient boundary $\Omega_V$, and let $\Delta_U$ be its complementary boundary.  Thus
\[
 \mu_*\Delta_U=\Omega_V-aS_V,\qquad
 L_U\sim_{\Q,\mu}K_U+\Delta_U,
\]
with $\Delta_U\geq0$ and exceptional coefficients less than one. The defining conditions, or functoriality, also identify this with the $K$-twisted transform for $h\mu$.  Every exceptional divisor for $h\mu$ either is exceptional for $\mu$, or is the transform of an $h$-exceptional divisor.  Its coefficient is less than one in both cases.  The Grauert--Riemenschneider theorem \cite[Theorem~3]{BK} over the two dlt pairs gives
\begin{equation}
 R^{>0}\mu_*\OO_U(L_U)=0,
 \qquad R^{>0}(h\mu)_*\OO_U(L_U)=0.
 \label{v-eq:resolutiongr}
\end{equation}

We also check the zeroth direct image for the plt ambient pair.  Let $F_U$ be the ordinary Fujita transform.  The difference $L_U-F_U$ is integral exceptional.  If $\Omega_U$ is the crepant boundary of $\Omega_V$, the comparison formula expresses this difference, up to relative linear equivalence, as
\[
 -\Omega_U+\Delta_U+\Delta_{U,L},
\]
where the last boundary is the ordinary Fujita boundary.  This expression is exceptional too, so two-sided negativity makes the relation an equality.  All its exceptional coefficients are greater than $-1$: the crepant plt coefficients are less than one and both Fujita boundaries are nonnegative.  Hence $L_U-F_U$ is effective.  Ordinary pushforward and the natural embedding into the divisorial sheaf downstairs yield $\mu_*\OO_U(L_U)=\OO_V(L_V)$.  Leray applied to \eqref{v-eq:resolutiongr} finishes the proof.  The argument includes $a=0$ and $a=1$.
\end{proof}

\medskip
We next study restriction to the centre on each intermediate model. The relevant quotient is a divisorial sheaf on the surface, so its vanishing must be checked for the divisor furnished by restriction.

\label{v-sec:quotients}

\begin{lemma}
\label{v-lem:surfacefano}
Let $g\colon V\to Z$ be a $\Q$-factorial Fano-type model from \cref{v-lem:survival}, with centre $T$.  Put $\Delta_T^0=\Diff_T(0)$.  There is an effective boundary $\Gamma_T\geq\Delta_T^0$ such that $(T,\Gamma_T)$ is klt and $-(K_T+\Gamma_T)$ is ample over $Z$.
\end{lemma}
\begin{proof}
The centre $T$ is normal by \cite[Theorem~19]{BK}.  Adjunction gives a klt pair
\[
 K_T+\Omega_T=(K_V+\Omega_V)|_T\sim_{\Q,Z}0,
 \qquad\Omega_T\geq\Delta_T^0.
\]
Choose a klt log Fano boundary $C$ on $V$, write $c=\coeff_T C<1$, and put $A_C=-(K_V+C)$.  Restricting \eqref{v-eq:pseudoeffective} shows that $-T|_T$ has an effective representative modulo relative $\Q$-linear equivalence.  Thus
\begin{equation}
 -(K_T+\Delta_T^0)\sim_{\Q,Z}
 A_C|_T+(C-cT)|_T+(1-c)(-T|_T)
 \label{v-eq:surfacebig}
\end{equation}
is relatively big.  The restrictions of effective divisors not containing $T$ are effective, after taking a Cartier multiple.

Consequently $Q_T=\Omega_T-\Delta_T^0$ is effective and relatively big. Choose $Q_T\sim_{\Q,Z}A_T+E_T$ with $A_T$ relatively ample and $E_T\geq0$.  For a sufficiently small rational $\eta>0$, the boundary
\[
 \Gamma_T=\Delta_T^0+(1-\eta)Q_T+\eta E_T
\]
is klt, by openness on a fixed surface log resolution, and its negative log divisor is relatively equivalent to $\eta A_T$.
\end{proof}

\begin{lemma}
\label{v-lem:quotient}
For every nef integral Weil divisor $H$ on $V$ there is an exact sequence
\begin{equation}
 0\to\OO_V(H-T)\to\OO_V(H)\to\OO_T(G_H)\to0,
 \label{v-eq:restriction}
\end{equation}
where $G_H$ is integral Weil and
\[
 G_H\sim_\Q H|_T-\delta_H,\qquad
 0\leq\delta_H\leq\Delta_T^0.
\]
Moreover $R^i(g|_T)_*\OO_T(G_H)=0$ for $i>0$.
\end{lemma}
\begin{proof}
Since $V$ is $\Q$-factorial and klt, every integral divisorial sheaf is Cohen--Macaulay by \cite[Theorems~3 and 19]{BK}, using zero correcting boundary in the Cohen--Macaulay criterion.  Their Lemma~5, applied to $(V,T)$, therefore gives the exact restriction sequence \eqref{v-eq:restriction} on $V$.  In its notation take $D_V=H-K_V-T$; adjunction turns its correction into $\delta_H=\Delta_T^0-\Theta_{D_V}$, with the asserted bounds.

For the boundary supplied by \cref{v-lem:surfacefano}, $\Gamma_T-\delta_H$ is effective klt and
\[
 G_H-(K_T+\Gamma_T-\delta_H)
 \sim_\Q H|_T-(K_T+\Gamma_T)
\]
is relatively ample.  If $g(T)$ is a point, $T$ is of del Pezzo type over the algebraically closed ground field, so \cite[Theorem~1.1]{ABL} gives the result for integral Weil divisors.  If $g(T)$ is a curve, factor through its normal Stein curve and use relative surface Kawamata--Viehweg vanishing, as recalled in \cite[Definition~6]{BK}. Finite pushforward to $Z$ is exact.  This applies before any connectedness or normality assertion for the image of $T$.
\end{proof}

\subsection{Vanishing along an ordered MMP}

The transport lemma applies when the coefficient of the centre lies between zero and one. We arrange the MMP in increasing order of its scaling parameter and use restriction sequences at the integer values of that parameter.

\label{v-sec:vanishing}

\begin{lemma}
\label{v-lem:scaling}
For a nef integral Weil divisor $D$ on $X$, a finite $-S$-MMP can be chosen with nondecreasing rational scaling values $r\geq0$.  At a step with scaling value $r$, the divisor $D_r-rS_r$ is nef and descends through the contraction.  Its final model $Y$ has $-T$ nef, and $D_Y-tT$ is nef for all $t$ at least the last scaling value.
\end{lemma}
\begin{proof}
By Fano type, $D$ is semiample \cite[Corollary~1.5]{ABL}. First run a $-S$-MMP over its semiample contraction, assigning scaling value zero. It terminates by \cref{v-lem:survival}, and $-S$ is then nef on every $D$-trivial ray. Rational polyhedrality of the relative cone implies that $D-tS$ is nef for small $t>0$. Increase $t$ to the endpoint $r$ of this interval. Unless $-S$ is already nef, an extremal ray satisfies
\[
 (-S_r)\cdot R<0,\qquad (D_r-rS_r)\cdot R=0.
\]
Contract it and flip if necessary. The semiample supporting divisor $D_r-rS_r$ descends and remains nef. Repeat at the same scaling value if another ray on its face is $-S$-negative; otherwise increase the parameter. All these steps form one klt MMP \eqref{v-eq:kltmmp}, so there are finitely many. Its endpoint has $-T$ nef, while $D_Y-rT$ is nef at the last scaling value. Adding nonnegative multiples of $-T$ proves the final assertion.
\end{proof}

\begin{theorem}[Integral vanishing]
\label{v-thm:vanishing}\label{main:24}
Let $f\colon X\to Z$ be a projective contraction from a $\Q$-factorial threefold of Fano type over $Z$, with $0<\dim Z<3$, in characteristic $p>5$. Suppose that a rational plt boundary $\Omega$ has irreducible vertical floor $S$ and $K_X+\Omega\sim_{\Q,Z}0$. For every nef integral Weil divisor $D$,
\begin{equation}
 R^if_*\OO_X(D-S)=0\qquad(i>0)
 \label{v-eq:mainvanishing}
\end{equation}
near $f(S)$.
\end{theorem}
\begin{proof}
Choose the MMP in \cref{v-lem:scaling}.  Let $M\geq1$ be an integer at least its last scaling value.  For $1\leq j\leq M$, let $X^{(j)}$ be the model after all steps with scaling value at most $j$, and put
\[
 H_j=D_j-jS_j.
\]
Each $H_j$ is an integral Weil divisor and is nef.  Also $X^{(M)}=Y$.

\emph{Vanishing on the final model.} Choose a klt log Fano boundary $C$ on $Y$, write $c=\coeff_T C$, and take a small rational $v>0$.  Then
\[
 \Gamma_Y=(1-v)\Omega_Y+v\bigl(C+(1-c)T\bigr)
\]
is plt with floor $T$, and
\[
 -(K_Y+\Gamma_Y)\sim_{\Q,Z}
 v\bigl(-(K_Y+C)-(1-c)T\bigr)
\]
is ample because $-T$ is nef.  Since $H_M$ is nef, $H_M-(K_Y+\Gamma_Y)$ is ample.  All the divisorial sheaves $\OO_Y(H_M-lT)$ are Cohen--Macaulay.  The surface vanishing hypothesis is verified by the log Fano adjunction on $T$ and the same integral-Weil argument as in \cref{v-lem:quotient}.  Thus the weak-pl vanishing proposition \cite[Proposition~7]{BK} gives
\begin{equation}
 R^{>0}f^{(M)}_*\OO_Y(H_M)=0
 \label{v-eq:terminalvanishing}
\end{equation}
near $f(S)$.

\emph{Inductive step.} On $X^{(j)}$ use
\begin{equation}
 0\to\OO_{X^{(j)}}(D_j-(j+1)S_j)
 \to\OO_{X^{(j)}}(H_j)
 \to\OO_{S_j}(G_{H_j})\to0.
 \label{v-eq:integersequence}
\end{equation}
Every intervening step on the way to $X^{(j+1)}$ has scaling value $r\in(j,j+1]$.  At such a step, the first divisor is
\begin{equation}
 D_r-(j+1)S_r=(D_r-rS_r)-(j+1-r)S_r,
 \qquad 0\leq j+1-r<1.
 \label{v-eq:wallcoefficient}
\end{equation}
The first parenthesis descends through the contraction.  Therefore \cref{v-lem:transport} identifies its derived pushforward to $Z$ with that of the corresponding divisor on the next model.  This works for divisorial steps as well as flips.  After all intervening steps the divisor is exactly $H_{j+1}$.  Its higher direct images vanish by induction.  Those of the quotient in \eqref{v-eq:integersequence} vanish by \cref{v-lem:quotient}.  The long exact sequence proves the assertion for $H_j$.  Descending induction from \eqref{v-eq:terminalvanishing} reaches $j=1$.

\emph{Return to the original sheaf.} Between $X$ and $X^{(1)}$ all steps have scaling value $r\in[0,1]$.  The transform of $D-S$ is
\[
 (D_r-rS_r)-(1-r)S_r,\qquad 0\leq1-r\leq1.
\]
Apply \cref{v-lem:transport}, including the parameter-zero steps, to identify its derived pushforward with that of $H_1$.  This proves \eqref{v-eq:mainvanishing}.
\end{proof}

\begin{corollary}
\label{v-cor:twovanishings}
Under the plt hypotheses of \cref{v-sec:target}, if $nB$ is integral, then
\[
 R^{>0}f_*\OO_X(-S)=0,\qquad
 R^{>0}f_*\OO_X(nN-S)=0
\]
near $f(S)$.
\end{corollary}
\begin{proof}
Apply \cref{v-lem:omega,v-thm:vanishing} with $D=0$ and $D=nN$ respectively. The two applications may use different MMPs.
\end{proof}

\subsection{Lifting a surface complement}

\label{v-sec:lifting}

\begin{proposition}[Lifting a prescribed surface complement]
\label{v-prop:lifting}
Under the plt hypotheses of \cref{v-sec:target}, put $K_S+B_S=(K_X+B)|_S$.  Suppose that $B_S^+\geq B_S$ is an $n$-complement over $z$ and $nB$ is integral.  Then $(X,B)$ has a monotonic $n$-complement over $z$ whose adjunction to $S$ is $B_S^+$.
\end{proposition}
\begin{proof}
The first vanishing in \cref{v-cor:twovanishings} identifies the Stein image of the normal surface $S$ with $f(S)$. Thus $S\cap X_z$ is connected. Adjunction gives an lc pair $(S,B_S)$, and the auxiliary boundary makes $S$ of Fano type over $f(S)$.

Use compatible canonical representatives, with coefficient $-1$ along $S$ and its transform $S'$ on a log resolution $\phi:X'\to X$. Write $K_{X'}+B'=\phi^*(K_X+B)$, put $T'=(B')^{=1}$, $\Delta'=B'-T'$ and $N'=\phi^*N$, and define
\begin{equation}
 Q'=n\Delta'-\floor{(n+1)\Delta'},\qquad
 L'=-nK_{X'}-nT'-\floor{(n+1)\Delta'}=nN'+Q'.
 \label{v-eq:liftdivisor}
\end{equation}
Take $\widehat\Gamma=(1-a)\Gamma+aB$ with $a<1$ close to one. The exceptional correction from \cref{gen:plt-lifting}, with nef part zero and $C'$ replaced by $S'$, gives effective reduced exceptional $P'$ such that
\begin{equation}
 \Lambda'=\widehat\Gamma'+Q'+P'
 \label{v-eq:lambda}
\end{equation}
is log smooth plt with floor $S'$. The coefficient calculation there is dimension-independent; $nB$ integral makes the correction exceptional.

For $D'=L'+P'-S'$ we have
\begin{equation}
 D'=K_{X'}+(\Lambda'-S')+
 \phi^*\bigl(-(K_X+\widehat\Gamma)+nN\bigr).
 \label{v-eq:liftgr}
\end{equation}
The zeroth-direct-image calculation in \cref{gen:plt-lifting} is dimension-independent: the exceptional coefficients of
\[
 D'-\phi^*(nN-S)=Q'+P'+\phi^*S-S'
\]
are greater than $-1$. It gives $\phi_*\OO_{X'}(D')=\OO_X(nN-S)$. The effective boundary $\Lambda'-S'$ has exceptional coefficients below one, so \cite[Theorem~3]{BK} gives the higher-direct-image vanishing. By \cref{v-cor:twovanishings} and Leray,
\begin{equation}
 R^1(f\phi)_*\OO_{X'}(L'+P'-S')=0.
 \label{v-eq:finalliftvanishing}
\end{equation}

Let $\psi:S'\to S$ and $R_S=B_S^+-B_S$. The prescribed complement section, pulled back to $S'$, has zero divisor
\begin{equation}
 G_{S'}=n\psi^*R_S+(Q'+P')|_{S'}.
 \label{v-eq:surfacesection}
\end{equation}
It is integral because it is a section of $(L'+P')|_{S'}$. It is effective by the same rounding calculation as in \cref{gen:plt-lifting}: a negative correction is greater than $-1$, and at the generic point of a curve on $S'$ only one other snc component contributes. The restriction sequence and \eqref{v-eq:finalliftvanishing} lift it to an effective $G'\sim L'+P'$ not containing $S'$. Put
\begin{equation}
 G=\phi_*G',\qquad B^+=B+G/n.
 \label{v-eq:complement}
\end{equation}
Then $G\sim nN$ and $n(K_X+B^+)\sim0$. Moreover $G'-P'-Q'-\phi^*G$ is exceptional and numerically trivial over $X$, hence zero. Restriction gives the exact adjunction identity $(K_X+B^+)|_S=K_S+B_S^+$. Inversion of adjunction \cite[Corollary~10.1]{BMP} proves lc near $S$.

If lc failed elsewhere on $X_z$, interpolate $B^+$ with $\Gamma$, with weight close to one on $B^+$. The resulting pair has ample negative log divisor, is plt near $S$, and is non-lc elsewhere on the fibre. Its non-klt locus in the fibre has the open-and-closed connected component $S\cap X_z$. By \cite[Theorem~1.2]{FW}, disconnectedness would force lc near that fibre, a contradiction. Thus $B^+$ is lc over $z$, and integrality of $nB$ gives the complement inequality.
\end{proof}

\medskip
We can now combine surface complements with the preceding lifting argument. The auxiliary plt boundary affects the construction, but does not enter the complement index.

\begin{proposition}
\label{bridge:plt-lifting}\label{main:44}\label{v-cor:lemma44}\label{v-cor:exceptional-adjoint}\label{v-cor:isolated-surface-complete}
Fix a finite rational set $J\subset[0,1]$. Let $f\colon X\to Z$ be a projective contraction, with $\dim X=3$ and $0<\dim Z<3$, over an algebraically closed field of characteristic $p>5$. Suppose that $X$ is $\Q$-factorial and of Fano type over $Z$, $(X,B)$ is lc, $\coeff B\subset J$, and $N:=-(K_X+B)$ is nef over $Z$. Suppose there is a rational plt boundary $\Gamma$ with irreducible vertical floor $S$, such that $S\subset\floor B$, $S\cap X_z\ne\varnothing$, and $-(K_X+\Gamma)$ is ample over $Z$. Then $(X,B)$ has a monotonic $n(J)$-complement over $z$.

\end{proposition}
\begin{proof}
By \cref{v-cor:twovanishings}, the two lifting vanishings hold once $nJ$ is integral. Adjunction gives an lc pair $(S,B_S)$ with coefficients in a hyperstandard set determined by $J$. The auxiliary boundary makes $S$ of Fano type over its normal Stein image, and the first vanishing identifies this image with $f(S)$. Apply \cref{v-thm:global-surface-complements,v-thm:relative-surface-complements} according as the image is a point or a curve. Taking a common multiple of that complement index and the denominators of $J$, use \cref{v-prop:lifting}. Its exceptional correction is chosen after $n$ and requires no denominator bound on $\Gamma$.
\end{proof}

\subsection{Reducing the threefold coefficients}

Before constructing a plt centre, we reduce the boundary to a finite coefficient set. The following rounding argument is uniform in the characteristic and preserves every coefficient-one component.

\begin{proposition}
\label{r-prop:rounding}
Let $\Phi\subset[0,1]$ be a DCC set.  There is $\epsilon=\epsilon(\Phi)>0$ with the following property in characteristic $p>5$.  Let $(X,B)$ be a projective threefold pair with a contraction $X\to Z$, $\dim Z>0$, such that
\begin{enumerate}[label=(\roman*)]
\item $X$ is $\Q$-factorial and of Fano type over $Z$;
\item $B\in\Phi\cup(1-\epsilon,1]$ and $-(K_X+B)$ is nef over $Z$;
\item there is $0\leq P\sim_{\Q,Z}-(K_X+B)$ such that $(X,B+P)$ is lc.
\end{enumerate}
Set
\[
       \Theta:=B^{\leq1-\epsilon}+\ceil{B^{>1-\epsilon}}.
\]
Run a $-(K_X+\Theta)$-MMP over $Z$ and denote its output by $X'$.  Then
\begin{enumerate}[label=(\alph*)]
\item $(X,\Theta)$ and $(X',\Theta')$ are lc;
\item no component of $\floor\Theta$ is contracted;
\item the MMP ends with $-(K_{X'}+\Theta')$ nef over $Z$.
\end{enumerate}
\end{proposition}

\begin{proof}
We use the argument of \cite[Proposition~2.50]{BirkarAnti}, with Waldron's ACC theorems \cite{Waldron} providing uniformity for $p>5$ and the inputs over imperfect function fields.

Put $D:=-(K_X+B)$.  Since $D$ is nef, it is a limit of movable divisors over $Z$. Because $X$ is of Fano type, every divisor MMP below exists and terminates.  Moreover, after every step the pushforward of $D$ is still a limit of movable divisors, the pushforward of $P$ is effective, and
\begin{equation}
                         K_X+B+P\sim_{\Q,Z}0.              \label{eq:display-4-3-1}
\end{equation}
On a common resolution of one MMP step, the two pullbacks of the transforms of \eqref{eq:display-4-3-1} agree.  Thus the transformed pair with boundary $B+P$ remains lc.  The $\Q$-factorial and Fano-type hypotheses are also preserved.  These observations justify using the same hypotheses at every intermediate model, with ``nef'' replaced by the weaker condition ``a limit of movable divisors''.  The proof of part~(a) below does not use nefness, so it applies in this weaker form on every such model.

\emph{(a) Log canonicity before the MMP.} If no uniform $\epsilon$ worked, there would be a sequence $\epsilon_i\downarrow0$ and pairs $(X_i,B_i)$ satisfying the hypotheses for which $(X_i,\Theta_i)$ is not lc.  Increasing the coefficients one at a time, choose a boundary
\[
 B_i\leq\Gamma_i\leq\Theta_i
\]
and a component $D_i$ whose coefficient $t_i\in(1-\epsilon_i,1)$ is the lc threshold of $D_i$ with respect to $(X_i,\Gamma_i-t_iD_i)$.  After passing to a subsequence, the union of all coefficients of the $\Gamma_i$ is DCC.  The strictly increasing sequence $t_i\to1$ contradicts Waldron's ACC for log canonical thresholds of excellent threefold pairs in residue characteristic different from $2,3,5$ \cite{Waldron}.  In particular, the choice of $\epsilon$ is uniform for all $p>5$. Thus $(X,\Theta)$ is lc.

\emph{(b) A rounded component cannot be contracted.} Suppose otherwise and take a sequence $\epsilon_i\downarrow0$.  Stop each MMP just before the first divisorial step $\pi_i\colon X_i\to Y_i$ which contracts a component $S_i\subset\floor{\Theta_i}$; a flip contracts no divisor.  Let $R_i$ be its extremal ray.  The transform $D_i^0:=-(K_{X_i}+B_i)$ is a limit of movable divisors, so $D_i^0\cdot R_i\geq0$: compute on a general curve in the exceptional divisor, avoiding the codimension-two stable base loci of the movable approximants.  Since $S_i\cdot R_i<0$, replacing only the coefficient of $S_i$ in $B_i$ by one, and calling the resulting boundary $\overline B_i$, gives
\[
 B_i\leq\overline B_i\leq\Theta_i,
 \qquad -(K_{X_i}+\overline B_i)\cdot R_i\geq0.
\]
Now interpolate the remaining coefficients from $\overline B_i$ towards $\Theta_i$. Since $-(K_{X_i}+\Theta_i)\cdot R_i<0$, there is a boundary $\overline B_i\leq\Gamma_i<\Theta_i$ such that
\[
 -(K_{X_i}+\Gamma_i)\cdot R_i=0,
\]
where $S_i\subset\floor{\Gamma_i}$, and some component $D_i$ has coefficient $t_i\in(1-\epsilon_i,1)$ and satisfies $D_i\cdot R_i>0$. By part~(a) in the limit-of-movable form just recorded, $(X_i,\Theta_i)$ is lc, and hence so is $(X_i,\Gamma_i)$. Write tildes for pushdown to $Y_i$.  Since $Y_i$ is $\Q$-factorial, write $K_{X_i}+\Gamma_i=\pi_i^*(K_{Y_i}+\widetilde\Gamma_i)+q_iS_i$. Intersecting with $R_i$ gives $q_i=0$.  Thus the pushdown is crepant and lc, and
\[
                         a(S_i,Y_i,\widetilde\Gamma_i)=0.
\]
Furthermore $\pi_i^*\widetilde D_i=D_i+\mu_iS_i$ with $\mu_i=-D_i\cdot R_i/(S_i\cdot R_i)>0$.  Hence, for every $\delta>0$,
\[
 a(S_i,Y_i,\widetilde\Gamma_i+\delta\widetilde D_i)
 =-\mu_i\delta<0.
\]
Thus $t_i$ is the lc threshold of $\widetilde D_i$ with respect to $(Y_i,\widetilde\Gamma_i-t_i\widetilde D_i)$.  Again $t_i\to1$ contradicts the three-dimensional ACC just cited.

\emph{(c) The output is not a Mori fibre space for the negative divisor.} If it were, after taking a sequence as above we would obtain contractions $X_i'\to T_i/Z$.  The transform of $-(K_{X_i}+B_i)$ is nef over $T_i$ (it is a limit of movable divisors and $\rho(X_i'/T_i)=1$), whereas $-(K_{X_i'}+\Theta_i')$ is anti-ample over $T_i$.  Interpolation therefore produces lc boundaries $B_i'\leq\Gamma_i'<\Theta_i'$ satisfying
\[
 K_{X_i'}+\Gamma_i'\num_{T_i}0,
\]
and a component ample over $T_i$ whose coefficient is $t_i\in(1-\epsilon_i,1)$.  It is horizontal over $T_i$. Part~(a), again in its limit-of-movable form on the output, shows that $(X_i',\Theta_i')$ is lc; therefore $(X_i',\Gamma_i')$ is lc as well.

Let $F_i$ be the generic fibre, over the possibly imperfect field $k(T_i)$.  Since $\dim Z>0$, one has $\dim F_i\leq2$.  The restricted pair is lc and numerically log Calabi--Yau, and the coefficient $t_i$ is retained.  In dimension two take a dlt modification of the generic fibre; the strict transform of the horizontal component still has coefficient $t_i$.  Waldron's global ACC for dlt surfaces over an arbitrary field of characteristic different from $2,3,5$ makes the possible $t_i$ finite \cite{Waldron}.  In dimension one the usual degree argument is the one-dimensional global ACC and gives the same conclusion.  Both contradict the strictly increasing sequence $t_i\to1$.  Hence the output has $-(K_{X'}+\Theta')$ nef.

Finally, the stability recorded in \eqref{eq:display-4-3-1}, together with preservation of the limit-of-movable condition, lets us apply the strengthened form of part~(a) on the output; hence $(X',\Theta')$ is lc.  This proves all three assertions.
\end{proof}

\subsection{Preparing a vertical plt centre}

\begin{proposition}
\label{bridge:pure-centre}
Let $f\colon(X,B)\to Z$ be an lc $\Q$-factorial threefold pair of Fano type over $Z$ in characteristic $p>5$, with $0<\dim Z<3$, $\coeff B\subset J$ finite, and $N:=-(K_X+B)$ nef over $Z$. Suppose that $\floor B$ contains a vertical prime $S$ meeting $X_z$. There is a crepant birational model $(Y,B_Y)$ of $(X,B)$, with $\coeff B_Y\subset J\cup\{1\}$, satisfying the hypotheses of \cref{bridge:plt-lifting}. In particular, its complement transfers to $(X,B)$ with the same index.
\end{proposition}
\begin{proof}
All constructions take place after shrinking the base about $z$.

\emph{Step 1. A nef and big divisor associated to the chosen centre.} Let $g\colon X\to V/Z$ be the semiample contraction of $N$, and put $D:=-(K_X+S)$. Fano type makes $D$ big over $Z$, because $S$ is vertical and $-K_X$ is big over $Z$. Over $V$ one has
\begin{equation}
 D\equiv_V B-S\geq0.\label{bridge:eq-pure}
\end{equation}
Run a $D$-MMP over $V$. Existence and termination follow from the Fano-type MMP; pseudo-effectivity excludes a Mori-fibre endpoint. Every step is crepant for $K_X+B\sim_{\Q,V}0$, so the transformed pair remains lc, its coefficients remain in $J$, and the transform of $N$ is pulled back from the same nef divisor on $V$.

The divisor $S$ is not contracted. At a first divisorial step contracting it, a general contracted curve covering $S$ would have negative $D$-degree. Equation \eqref{bridge:eq-pure} holds on that model, but the effective divisor $B-S$ does not contain $S$ and has nonnegative degree on this curve. This is a contradiction. Flips lose no divisors. The output is again $\Q$-factorial and of Fano type over $Z$; rename it $X$. Now $D$ is nef over $V$ and remains big over $Z$.

The relative cone of curves of a Fano-type model is rational polyhedral. On each of its finitely many rays on which $N$ is zero, $D$ is nef. Consequently $mN+D$ is nef over $Z$ for all sufficiently large integers $m$. Set
\[
 \Delta:=\frac{m}{m+1}B+\frac1{m+1}S,
 \qquad N_\Delta:=-(K_X+\Delta)=\frac{mN+D}{m+1}.
\]
Then $N_\Delta$ is nef and big over $Z$, $\Delta\leq B$, and $\floor\Delta=S$. Every lc centre of $(X,\Delta)$ is vertical. Indeed, an lc place of this convex combination is an lc place of $(X,S)$. Since $X$ is klt, its centre is contained in $S$. This holds in either dimension of the base.

Take a crepant $\Q$-factorial dlt modification of $(X,\Delta)$. Its exceptional divisors are lc places of $(X,B)$, since $\Delta\leq B$ and both pairs are lc. Thus the crepant transform $B_Y$ is effective, with coefficients in $J\cup\{1\}$, and its negative log divisor is nef. All components of the dlt floor are vertical. The model is of Fano type over $Z$: combine the crepant sub-boundary of a klt weak log Fano boundary on $X$ with the crepant dlt boundary, with coefficient sufficiently close to one on the latter. This is an effective klt boundary with nef and big negative log divisor. Rename the model and the two boundaries $X,B,\Delta$.

\emph{Step 2. An ample lc boundary with controlled lc places.} Choose an effective klt log Fano boundary $C$ on $X$. For some rational $\tau>0$ one can write
\[
 N_\Delta-\tau C\sim_{\Q,Z}A+G,
 \qquad A\text{ ample},\quad G\geq0.
\]
Choose $0<\varepsilon<\tau/2$ small enough that $A+\varepsilon C$ is ample, and put
\[
 \widetilde\Delta=(1-\varepsilon)\Delta+\varepsilon C,
 \quad G_\varepsilon=G+(\tau-\varepsilon)C,
 \quad H=G_\varepsilon+\varepsilon\Delta.
\]
The pair $(X,\widetilde\Delta)$ is klt, its negative log divisor $N_{\widetilde\Delta}$ is nef and big, and $N_{\widetilde\Delta}\sim_{\Q,Z}A+H$.

If $\Supp G_\varepsilon$ contains no lc centre of $(X,\Delta)$ meeting $X_z$, then for small $\delta>0$ the pair $(X,\Delta+\delta G_\varepsilon)$ is dlt and its negative log divisor is ample, since it is equivalent to
\[
 (1-\delta)N_\Delta+\delta(A+\varepsilon C).
\]
Retain one floor component meeting $X_z$ and lower all the others slightly. This gives the required plt boundary $\Gamma$.

Otherwise fix an lc place $F$ of $(X,\Delta)$ whose centre is contained in $\Supp(G+C)$ and meets $X_z$. Set
\[
 u=\lct_z(X,\widetilde\Delta;H),
 \qquad\Omega=\widetilde\Delta+uH.
\]
The numerator $a(F,X,\widetilde\Delta)$ equals $\varepsilon a(F,X,C)$, whereas $\operatorname{ord}_F H$ is bounded below by a positive constant as $\varepsilon\to0$. Thus $u\to0$, $0<u<1$, and $\Omega\to\Delta$ coefficientwise. Moreover
\[
 -(K_X+\Omega)\sim_{\Q,Z}
 (1-u)N_{\widetilde\Delta}+uA
\]
is ample.

On one fixed log resolution of $\Supp(\Delta+G+C)$, take the minimum of the finitely many positive gaps to one in the pullback of $\Delta$. For sufficiently small $\varepsilon$, every coefficient-one component in the pullback of $\Omega$ already has coefficient one for $\Delta$. The lc places above a log smooth sub-lc pair are the toroidal places over coefficient-one strata. Therefore every lc place of $(X,\Omega)$ is an lc place of $(X,\Delta)$, hence of $(X,B)$, and is vertical.

If $\floor\Omega\ne0$, choose a component $T$ meeting $X_z$. It belongs to the dlt floor of $\Delta$, so $(X,T)$ is plt. For small $\eta>0$, the boundary $(1-\eta)\Omega+\eta T$ is plt with floor $T$ and still has ample negative log divisor. This again gives the required $\Gamma$.

\emph{Step 3. A plt extraction.} If $\floor\Omega=0$, take a dlt modification $\rho\colon(X',\Omega')\to(X,\Omega)$ and write $F=\floor\Omega'=\Exc(\rho)_{\rm red}$. Run a $(K_{X'}+F)$-MMP over $X$. Since $X$ is klt, the divisor
\[
 K_{X'}+F-\rho^*K_X
\]
is effective exceptional and has strictly positive coefficient on every $\rho$-exceptional prime. Its support on each model is precisely the remaining exceptional divisors. Negativity excludes any of them on a nef output. The final morphism is an isomorphism, since the base $X$ is $\Q$-factorial. Thus immediately before the last step there is a single-divisor extraction $h\colon Y\to X$, with exceptional divisor $T$, such that $(Y,T)$ is plt and $-(K_Y+T)$ is $h$-ample.

The divisor $T$ is a vertical lc place of $(X,\Omega)$ and hence of $(X,B)$. Consequently
\[
 K_Y+B_Y=h^*(K_X+B),\qquad B_Y=B^\sim+T\geq0.
\]
Write $K_Y+\Omega_Y=h^*(K_X+\Omega)$. For small rational $\lambda>0$, the boundary $\Gamma_Y=(1-\lambda)\Omega_Y+\lambda T$ is plt with floor $T$ and
\[
 -(K_Y+\Gamma_Y)
 =(1-\lambda)h^*(-(K_X+\Omega))
   +\lambda(-(K_Y+T))
\]
is ample over $Z$. Lowering $T$ a little shows that $Y$ is of Fano type over $Z$. The negative log divisor of $B_Y$ is the nef pullback of that of $B$, and its coefficients belong to $J\cup\{1\}$. All comparisons used for $B$ were crepant, completing the proof.
\end{proof}

\subsection{Relative complements}

\begin{proof}[Proof of \cref{bridge:relative-complements}]
Every point of $Z$ has a closed specialization, so it suffices to treat closed points. Fix such a point $z$ and work on an affine neighbourhood of it. Choose an effective Cartier divisor $H$ through $z$ and let $t=\lct_z(X,B;f^*H)$. Take a dlt modification of $(X,B+tf^*H)$ which displays a vertical lc place meeting $X_z$, including when $t=0$. On it take the strict transform of $B$, put coefficient one on the exceptional divisors, and raise the chosen floor component to one. Call this boundary $\Theta$. It has coefficients in $\Phi(R)\cup\{1\}$, lies below the crepant pullback of $B+tf^*H$, and lies above the crepant sub-boundary of $B$.

This dlt model is of Fano type: the convex-combination argument in \cref{bridge:pure-centre} applies with the lc boundary $B+tf^*H$, whose negative log divisor is relatively nef. The divisor $-(K+\Theta)$ is the sum of that nef divisor and an effective divisor not containing the chosen floor. Complete the nef divisor by a sufficiently diluted semiample effective divisor so that the resulting pair is lc. A $-(K+\Theta)$-MMP therefore terminates at a nef model; the same effective representative proves that the chosen floor survives. The lc log Calabi--Yau completion is crepant throughout. Thus the resulting model satisfies the hypotheses needed for rounding.

Apply \cref{r-prop:rounding}. It produces a nef lc model with a surviving vertical floor and coefficients in a finite set $J=J(R)$. All complements transfer back by \cref{r-lem:subcrepant}. Apply \cref{bridge:pure-centre} to this model, followed by \cref{bridge:plt-lifting}. Only $J\cup\{1\}$ enters the resulting index. Take its common multiple with the original seed denominators and reverse the crepant and sub-crepant comparisons. The complement rounding inequality is preserved by the hyperstandard arithmetic in \cref{s-lem:rounding,s-prop:complement-arithmetic}.
\end{proof}

\section{Global complements for fibrations}
\label{sec:fibered-complements}

To obtain a global complement, we descend to the base of the fibration and complement there. For a surface base, the generic fibre is a genus-zero curve; its marked points determine the moduli divisor. For a curve base, we first control the geometry of the generic surface. In both cases an effective canonical bundle formula bounds the base coefficients and the Cartier index of the moduli part.

\subsection{Geometric properties of bounded surface fibres}

For the curve base case, geometric log canonicity is obtained from boundedness. We first give the two family-theoretic statements used in that passage.

\begin{lemma}
\label{lem:bounded-geometric-normality}
Let $T$ be a scheme of finite type over $\operatorname{Spec}\mathbb Z$ and let $h:\mathcal X\to T$ be projective and of finite presentation. There is an integer $N>0$, depending only on this family, such that the following holds. If $K$ is a field with $\operatorname{char}K=0$ or $\operatorname{char}K\nmid N$, and $t:\operatorname{Spec}K\to T$ is a point for which $X_t:=\mathcal X\times_T\operatorname{Spec}K$ is geometrically integral and normal, then $X_t$ is geometrically normal over $K$.
\end{lemma}

\begin{proof}
By flat stratification, openness of geometric integrality \cite[IV$_3$, Theorem~12.2.4(viii)]{EGAIV} and noetherian induction, reduce to a flat projective family with geometrically integral fibres over an integral base $S$ of finite type over $\mathbb Z$. Suppose $S$ dominates $\Spec\mathbb Z$. Normalize its integral total space: $\nu:Y\to X$. The generic normalization is geometrically normal and geometrically integral, since its field has characteristic zero. After shrinking $S$, the family $Y\to S$ is flat with geometrically normal and geometrically integral fibres \cite[IV$_3$, Theorem~12.2.4(iv),(viii)]{EGAIV}.

Shrink again so that the isomorphism locus of $\nu$ meets every fibre. For any field-valued point over this open, $Y_K\to X_K$ is then finite birational between geometrically integral schemes. If $X_K$ is normal, it is an isomorphism, so $X_K$ is geometrically normal. Apply the same argument to the reduced closed complement. Noetherian induction gives finitely many such strata and finitely many strata supported in positive characteristics. Invert the product of those characteristics.
\end{proof}

\begin{lemma}
\label{lem:bounded-log-smooth-resolution}
Let $(\mathcal X,\mathcal D)\to T$ be a fixed projective family of finite presentation over a scheme of finite type over $\mathbb Z$, where $\mathcal D\subset\mathcal X$ is a closed subscheme. After excluding finitely many characteristics, every field-valued member $(F,D)$ with $F$ a normal geometrically integral surface and $D_{\mathrm{red}}$ a divisor admits a projective log resolution $\pi:(Y,E)\to(F,D_{\mathrm{red}})$ with $Y$ smooth over its ground field and $E$ a simple normal crossing divisor whose components and intersection strata are smooth over that field.
\end{lemma}

\begin{proof}
By \cref{lem:bounded-geometric-normality}, flat stratification and noetherian induction, it suffices to work over an integral characteristic-zero stratum $S$ with geometrically normal and geometrically integral fibres. If the generic $\mathcal D$ contains the generic fibre, discard the corresponding open, which has no relevant members. Otherwise resolve the generic couple over $\kappa(\eta)$ by \cite[Theorems~1.0.1--1.0.3]{Wlodarczyk}. Enlarge its snc divisor to contain the complement of an open isomorphism locus disjoint from $\mathcal D_\eta$.

Spread this projective morphism, the open isomorphism and the finitely many boundary strata over an open of $S$. After shrinking, the source is smooth with geometrically integral fibres, the boundary components are relative Cartier divisors with smooth intersections of the expected codimension, and the isomorphism locus meets every fibre. The exceptional locus and inverse image of $\mathcal D$ remain in the boundary. Each fibre is therefore the required log resolution. Noetherian induction on the closed complement gives finitely many resolution strata. Excluding their vertical characteristics proves the assertion.
\end{proof}

\begin{lemma}
\label{a-lem:surface-fibre}
Fix $\epsilon>0$ and a finite rational seed $R\subset[0,1]$. There is $p_{\mathrm{sf}}=p_{\mathrm{sf}}(\epsilon,R)$ with the following property. Let $f:(X,B)\to C$ be an lc-trivial contraction from a projective lc threefold pair over an algebraically closed field of characteristic $p>p_{\mathrm{sf}}$ to a smooth curve. Suppose $\operatorname{coeff}(B^h)\subset\Phi(R)$ and the underlying generic surface $F=X_\eta$ is $\epsilon$-lc and Fano. Then $(F_{\overline{k(C)}},B_{F,\overline{k(C)}})$ is lc.
\end{lemma}

\begin{proof}
Put $K=k(C)$ and $B_F=B|_F$. Then $(F,B_F)$ is lc and $K_F+B_F\sim_{\mathbb Q}0$. Waldron's surface global ACC \cite[Theorem 1.4]{Waldron} over arbitrary fields, applied after a dlt modification, puts the coefficients of $B_F$ in a finite rational set $I=I(R)$. Choose $0<\epsilon_0<\epsilon$. Bernasconi--Martin's boundedness \cite[Theorem 5.12]{BernasconiMartin} of $\epsilon_0$-klt del Pezzo surfaces over arbitrary fields of characteristic different from $2,3,5$ gives a uniform integer $r$ such that $H=-rK_F$ is Cartier and very ample, with $H^2$ and $h^0(F,H)$ bounded; these surfaces are geometrically integral. If $D_F=\operatorname{Supp}B_F$ and $\delta$ is the least positive element of $I$, then
\[
 H\cdot D_F\leq\delta^{-1}H\cdot B_F
              =H^2/(r\delta).
\]
Thus the couples $(F,D_F)$ lie in a fixed projective family over a scheme of finite type over $\mathbb Z$, depending only on $\epsilon,R$.

Apply \cref{lem:bounded-geometric-normality,lem:bounded-log-smooth-resolution}. After increasing the characteristic bound, $F$ is geometrically normal and there is a projective resolution $\pi:(Y,E)\to(F,D_F)$ with $(Y,E)$ log smooth over $K$. Write
\[
 K_Y+B_Y=\pi^*(K_F+B_F).
\]
The support of $B_Y$ is contained in $E$, and its coefficients are at most one because $(F,B_F)$ is lc. Choose any positive integer $m$, depending on this member, for which $m(K_F+B_F)$ is Cartier. Flat field base change for the dualizing sheaves and this Cartier pullback identity give, for every field extension $L/K$,
\[
 K_{Y_L}+(B_Y)_L
       =\pi_L^*(K_{F_L}+(B_F)_L).
\]
Here $F_L$ is normal by geometric normality, and $(Y_L,E_L)$ is snc with reduced components and strata. The coefficients in this equality are therefore unchanged, so they are at most one. Taking $L=\overline K$ proves the assertion.

\end{proof}

\begin{lemma}
\label{lem:geonormal-contraction}
Let $h:Y\to Z$ be a proper birational morphism of normal integral varieties over a field $K$. If $Y$ is geometrically integral and geometrically normal, then $Z$ is geometrically integral and geometrically normal.
\end{lemma}
\begin{proof}
The two varieties have the same function field, so geometric integrality descends. For any extension $L/K$, flat base change \cite[Lemma~30.5.2]{Stacks} gives $(h_L)_*\OO_{Y_L}=\OO_{Z_L}$. Since $Y_L$ is normal, $h_L$ factors through the finite normalization $\widetilde Z_L\to Z_L$. Proper birational pushforward to the normal $\widetilde Z_L$ then gives $\OO_{\widetilde Z_L}=\OO_{Z_L}$. Thus $Z_L$ is normal.
\end{proof}

\begin{proposition}
\label{prop:controlled-gglc-dlt}
Let $f:(X,B)\to C$ be an lc-trivial contraction from a projective lc threefold pair over an algebraically closed field of characteristic $p>5$ to a smooth curve. Suppose that the generic pair admits a projective log resolution $\pi_\eta:(Y_\eta,E_\eta)\to(X_\eta,\operatorname{Supp}B_\eta)$ which is log smooth over $K=k(C)$. Then there is a projective crepant $\mathbb Q$-factorial dlt modification $g:(X',B')\to(X,B)$ such that the geometric generic pair $(X'_{\overline\eta},B'_{\overline\eta})$ is dlt.
\end{proposition}

\begin{proof}
Spread the chosen generic resolution over an open of $C$, take a projective closure over $X$, and resolve the remaining pair while preserving this log-smooth open \cite[Theorem~1.4]{Benozzo}. We obtain $\pi:Y\to X$ with the prescribed generic fibre. Set $\Delta=\pi^{-1}_*B+\Exc(\pi)_{\rm red}$. Then
\[
 K_Y+\Delta=\pi^*(K_X+B)+A,\qquad
 A=\sum_{E_i\subset\Exc(\pi)}a(E_i,X,B)E_i\geq0.
\]
Run the $(K_Y+\Delta)$-MMP over $X$ \cite[Theorems~1.1--1.2]{YYZ}. A divisorial step contracts only a component of $A$: an effective representative excluding a contracted prime has nonnegative degree on a general contracted curve covering it. At the nef endpoint negativity makes the remaining exceptional $A$ zero. Thus the output $g:(X',B')\to(X,B)$ is crepant, $\Q$-factorial dlt, and every contracted exceptional prime had positive original discrepancy.

A flip is an isomorphism on the generic surface, since a normal surface has no nontrivial small contraction. Hence the generic MMP induces a morphism $h:Y_\eta\to F'=X'_\eta$. The smooth proper $Y_\eta$ has $H^0(Y_\eta,\OO)=K$, so it is geometrically integral. By \cref{lem:geonormal-contraction}, $F'$ is geometrically normal. Write
\[
 K_{Y_\eta}+\Gamma_\eta=h^*(K_{F'}+B'_{F'}).
\]
Every $h$-exceptional coefficient is below one, and its support, together with $\Supp\Gamma_\eta$, lies in the chosen log-smooth divisor. These properties persist after every field extension: the source and boundary strata are smooth, the target is geometrically normal, and canonical pullback commutes with flat field extension after taking any Cartier multiple. The exceptional curves are the geometric components of the original exceptional curves. The log-resolution criterion therefore makes $(F'_L,B'_{F',L})$ dlt for every $L/K$.
\end{proof}

\subsection{Threshold canonical bundle formulas}

We now consider curve fibres over a surface base. The first lemma ensures that, in a characteristic depending only on the coefficient seed, the original marked points are separable. We then compute the thresholds after a perturbation whose residual cycle has total weight less than one.

\begin{lemma}
\label{a-lem:curve-fibre}
Let $f\colon(X,B)\to Z$ be a contraction with $\dim X=3$, $\dim Z=2$, $(X,B)$ lc, $K_X+B\sim_\Q0/Z$, and $B$ big over $Z$.  Assume that the coefficients of $B^h$ belong to $\Phi(R)$.  Let
\[
 \delta_R:=\min\bigl(\Phi(R)\cap(0,1]\bigr).
\]
If $p>\max\{3,2/\delta_R\}$, then the geometric generic fibre is $\mathbb P^1$, every horizontal component of $B$ is generically separable over $Z$, and the geometric generic log fibre is lc.
\end{lemma}

\begin{proof}
Put $K=k(Z)$ and $F=X_\eta$.  Since $B$ is relatively big and $K_F+B_F\sim_\Q0$, one has $\deg K_F=-\deg B_F<0$.  Witaszek's smoothness criterion \cite[Proposition~2.2]{WitaszekCBF} shows that $F$ is a smooth genus-zero curve over $K$ for $p>2$.  Connectedness gives $H^0(F,\mathcal O_F)=K$, and hence $F_{\bar K}\simeq\mathbb P^1_{\bar K}$.

Let $D$ be a horizontal prime divisor with positive coefficient $b_D$, and write $p^{e_D}$ for the inseparable degree of $K(D)/K(Z)$.  On the geometric generic fibre, the scheme-theoretic restriction of $D$ has multiplicity $p^{e_D}$.  Since
\[
 2=\deg B_{\bar F}\ge p^{e_D}b_D,
\]
the inequalities $b_D\ge\delta_R$ and $p>2/\delta_R$ force $e_D=0$. Thus every horizontal component is separable.  Its restriction to $F_{\bar K}$ is therefore a reduced set of points; restrictions of distinct horizontal primes are disjoint.  The coefficient at each such point is at most one because $(X,B)$ is lc at the generic point of the horizontal prime.  Hence $(F_{\bar K},B_{\bar F})$ is lc.
\end{proof}

\begin{lemma}
\label{cbf:finite-invariants}
Let $a_i,c_j>0$, for $1\leq i\leq N$ and $1\leq j\leq s$, be rational numbers, and let $d_j$ be positive integers such that $\sum_i a_i+\sum_j c_jd_j=2$. Choose $m$ clearing all weights and put
\[
 \mathcal P=(\PP^1)^N\times\prod_j\PP^{d_j},\qquad
 \mathcal A=\boxtimes_i\OO(ma_i)\boxtimes_j\OO(mc_j),\qquad
 q=\prod_jd_j!.
\]
The factors $\PP^{d_j}$ parameterize effective degree-$d_j$ cycles on $\PP^1$. There are finitely many sections $I_\ell$ of $\mathcal A^q$ with the following properties. Their common zero locus consists exactly of configurations having a geometric point of total weight greater than one. Under simultaneous substitution by $g\in\mathrm{GL}_2$ in the defining linear and binary forms, each $I_\ell$ is multiplied by $(\det g)^{mq}$.
\end{lemma}
\begin{proof}
Temporarily label the roots of each binary form. Replace a label of weight $b$ by $mb$ copies of its defining vector; there are $2m$ copies in total. Pair the copies and multiply the determinants of the paired vectors to obtain a bracket monomial. If each coincident group has at most $m$ copies, arrange the groups consecutively around $2m$ circular positions and pair opposite positions. Every determinant is nonzero. Conversely, if a group has more than $m$ copies, every pairing contains two vectors from that group and its determinant vanishes.

Let $M_\alpha$ run through the finitely many such monomials, and let $H=\prod_j\mathfrak S_{d_j}$ permute the root labels. Expand
\[
 \prod_{h\in H}h\Bigl(\sum_\alpha t_\alpha M_\alpha\Bigr)
       =\sum_\ell t^{\nu_\ell}I_\ell.
\]
The coefficients are symmetric sections of the prescribed root multidegrees multiplied by $q$. They descend through $\operatorname{Sym}^{d_j}(\PP^1)=\PP^{d_j}$ to $\mathcal A^q$. At a configuration with all point weights at most one, every factor is a nonzero polynomial in the $t_\alpha$, so some coefficient is nonzero. At a heavier point all coefficients vanish. Each monomial has $m$ brackets, so the product has determinant weight $mq$. No division by $|H|$ is involved; the construction holds in every characteristic, including for nonreduced cycles.
\end{proof}

\begin{lemma}
\label{cbf:anchored-dvr}
Let $R$ be the DVR at a prime of a smooth surface, with parameter $u$ and fraction field $K$. On $\PP^1_R$ let
\[
 C=\sum_i a_iS_i+\sum_jc_jD_j,\qquad
 \sum_i a_i+\sum_jc_jd_j=2,\qquad w:=\sum_jc_jd_j<1,
\]
where $a_i,c_j$ are positive rational numbers, the $S_i$ are distinct $K$-rational sections including the sections at $0$ and $\infty$, and the $D_j$ are effective horizontal cycles of generic degree $d_j$ containing none of the $S_i$. Assume that every geometric generic point weight is at most one. Evaluate the sections of \cref{cbf:finite-invariants} on defining forms whose coefficients include a unit, and let $J$ be the ideal generated by these evaluations in $R$. Take the threshold over real parameters, allowing negative values, with sub-lc required near the special fibre. Then
\begin{equation}
 \lct(\PP^1_R,C;\{u=0\})
       =1-\frac{\ord_RJ}{mq}.                 \label{eq:anchored-threshold}
\end{equation}
\end{lemma}
\begin{proof}
In an affine or inverse chart containing a marked section, take its coordinate $a\in R$. For a Gauss valuation $v(u)=e$, $v(x-a)=er$, with $r\geq0$ rational, put $h(r)=v(C)/e$. Its log discrepancy divided by $e$ is $1+r$, so its threshold bound is
\[
 \theta(r)=1+r-h(r).
\]
Organize these valuations by inclusion of their discs. The paths joining the marked sections form a tree containing the original Gauss vertex, because $0,\infty$ are marked sections. Along a branch the slope of $\theta$ is one minus the total weight in that branch. This can be checked by factoring the forms over a valued algebraic closure and taking the valuations of their linear factors. Every branch leaving the tree spanned by the marked sections has weight at most $w<1$, so the minimum lies on that tree. Terminal rays have nonnegative slope, since each generic point weight is at most one. Thus a minimum $\theta_0$ is attained at a rational radius. After translating and scaling to this valuation, every geometric reduction point has weight at most one: a heavier branch would decrease $\theta$. This argument uses the algebraic closure only to compute valuations of polynomials.

We check that no other valuation lowers the threshold. Separate the marked sections, and then separate them from the strict transforms of the residual cycles, by blowing up their intersection points along the marked sections. These operations terminate because their finite intersection lengths decrease. A section has intersection one with the total fibre, so it always meets one multiplicity-one component transversely away from the vertical nodes. All these blowups therefore follow directions of the marked sections. At every remaining vertical node take a regular toric subdivision containing the finitely many compact Newton slopes of the equations of the residual cycles. The strict transforms of the residual cycles then avoid every vertical node; no Newton nondegeneracy is required. These subdivisions only subdivide edges of the same tree.

On this proper regular model write the total fibre as $F=\sum_\nu e_\nu E_\nu$ and write $C_{\rm res}$ for the strict transform of the residual boundary. Its horizontal components remain finite over $R$. Consequently
\[
 w=F\cdot C_{\rm res}
   =\sum_\nu e_\nu(E_\nu\cdot C_{\rm res}),\qquad
 E_\nu\cdot C_{\rm res}\geq0.
\]
Every local coefficient of $C_{\rm res}|_{E_\nu}$ is therefore at most $w/e_\nu<1$, including all multiplicities and residue degrees. The crepant vertical coefficients for $C+\theta_0\{u=0\}$ are at most one by the Gauss bounds. At a vertical node, or where a marked section meets the fibre, the resulting pair is snc with coefficients at most one. At any other point, raise the vertical coefficient to one. Surface inversion of adjunction \cite[Theorem~5.1]{TanakaMMP} applies to the smooth component $E_\nu$ over the excellent regular base $\Spec R$, and proves lc there. Lowering the coefficient preserves lc. This single model therefore controls all further valuations and proves that the threshold equals $\theta_0$.

At a minimizing radius $r$, substitute $x=a+u^rz$ after a finite DVR extension permitting this scale, and make each transformed form primitive by dividing by $u^{\eta_i}$. Include the marked sections among the forms. Their weighted normalization order is $h=\sum b_i\eta_i$. The determinant character in \cref{cbf:finite-invariants} gives
\[
 \ord I_\ell(\mathrm{transformed})
  =\ord I_\ell(\mathrm{initial})+mq(r-h)
  =\ord I_\ell(\mathrm{initial})+mq(\theta_0-1).
\]
The transformed coefficient vectors are integral and their geometric reduction has point weights at most one. Hence some $I_\ell$ is a unit, while all are integral. Taking the minimum proves \eqref{eq:anchored-threshold}, with valuations normalized by $\ord(u)=1$. The extension is used only for this polynomial identity; no threshold descent through it is required.
\end{proof}

The preceding threshold calculation gives a canonical bundle formula for the following class of genus-zero fibrations. The residual cycle need not split over the tame extension.

\begin{theorem}[Genus-zero formula with tame marked points]
\label{a-lem:genus-zero-cbf}\label{cbf:anchored}
Let $f:X\to Z$ be a projective contraction from a normal threefold to a projective normal surface in characteristic $p>5$, and let $B$ be a rational divisor with $K_X+B$ $\Q$-Cartier. Assume $K_X+B\sim_\Q f^*D$ for a $\Q$-Cartier divisor $D$ on $Z$, the geometric generic fibre is $\PP^1$, and the geometric generic log pair is lc. Suppose its boundary is the sum of an effective marked divisor splitting over a Galois extension of $k(Z)$ of degree prime to $p$, and an effective remaining cycle of total geometric weighted degree less than one. Then the threshold moduli b-divisor is b-semiample, hence b-nef and b-$\Q$-Cartier. Negative vertical coefficients on crepant models are allowed.
\end{theorem}
\begin{proof}
Normalize the base in the stated Galois extension. The induced total function-field extension also has degree prime to $p$, so every divisorial extension is tame. For the absolute crepant boundary, log discrepancies scale by the total ramification index. Comparing orders of base parameters gives
\[
 t_{P'}=e(P'/P)t_P,\qquad
 K_{Z'}+B_{Z'}=\tau^*(K_Z+B_Z),\qquad
 M_{Z'}=\tau^*M_Z.
\]
These identities hold over every higher base model; the residual cycle is not split. This is the finite-cover discrepancy calculation \cite[Proposition~99 and Corollary~100]{KollarCovers}.

Work first over the cover. There are at least two distinct marked points, since their total weight is greater than one and each point weight is at most one. Choose two as $0,\infty$ and resolve the base to $T$. The generic curve is now split. Replace the total space by $\PP^1\times T$ crepantly. This is legitimate: a fixed rational trivialization of a multiple of $K_X+B-f^*D$ defines its pullback on every birational model, so pushing it to this product and pulling to a common model gives the same divisor. Its boundary has the form $C+f^*V$, with $C$ horizontal and $V$ possibly negative.

Resolve the rational maps defined by the marked points and binary forms, obtaining $\vartheta:T\to\mathcal P$ as in \cref{cbf:finite-invariants}. The universal horizontal cycles are finite and flat. Put $\mathcal L=\vartheta^*\mathcal A/m$ and let $J$ be the base ideal of the pulled-back finite invariant system. Fix rational frames of these parameter lines and the same generic log trivialization throughout. The universal polynomial product $G$ then gives an identity of divisors
\[
 m(K_{\PP^1_T/T}+C)=f^*(m\mathcal L)+\Div(G).
\]
Comparison with the original trivialization yields $D_T-K_T=\mathcal L+V+A_0$, where $A_0$ is one fixed principal $\Q$-divisor from $k(T)^*$. On a higher model $\pi:T'\to T$, the vertical part becomes $\pi^*V-(K_{T'}-\pi^*K_T)$, and the other terms pull back. Thus these normalizations are simultaneous on all models.

Apply \cref{cbf:anchored-dvr} at every divisorial valuation of every higher base model. The vertical term shifts the threshold by $-V_P$, so the b-divisor identities are
\[
 B_P=V_P+\frac{\ord_PJ}{mq},\qquad
 \mathbf M=\overline{\mathcal L+A_0}-\frac1{mq}Z(J),
\]
where $Z(J)$ is the Cartier b-divisor of the ideal. Principalize $J$. After removing its common divisor, the invariant sections generate $\vartheta^*\mathcal A^q\otimes\OO(-Z(J))$. Hence $mqM$ is $\Q$-linearly equivalent to a globally generated divisor; clearing also the fixed trivialization denominator gives a generated multiple. The identity at all higher valuations proves b-Cartierness and b-semiampleness.

Finally take a common equivariant refinement of this model and its Galois translates and form the normal surface quotient $T/G=W$. The finite identities above identify $M_T$ with $\tau^*M_W$. A Cartier multiple of $M_T$ has a rational local equation near each finite orbit; its field norm makes the corresponding $|G|$-multiple of $M_W$ Cartier. Norms of sections prove semiampleness: a globally generated multiple upstairs has a section nonzero at every point of a chosen finite orbit, and its norm is nonzero downstairs. The same finite identities on higher models give Cartier closure downstairs. Thus the threshold moduli b-divisor on $Z$ is b-semiample.
\end{proof}

To prove Fano type of the base, we will replace a generalized klt pair by an ordinary klt boundary.

\begin{lemma}
\label{a-lem:ordinary-approx}
Let $(Y,B_Y+\mathbf M_Y)$ be a projective generalized klt surface pair, where $B_Y\ge0$ and $\mathbf M$ is a b-nef b-$\Q$-Cartier b-divisor.  For every ample $\Q$-divisor $A$ on $Y$, there is an effective $\Q$-divisor $\Omega$ such that
\begin{equation}
 (Y,\Omega)\text{ is klt},\qquad
 \Omega\sim_\Q B_Y+M_Y+A.                     \label{eq:display-3-9}
\end{equation}
\end{lemma}

\begin{proof}
On a log resolution $\phi:W\to Y$ where the nef part descends, write $K_W+B_W+M_W=\phi^*(K_Y+B_Y+M_Y)$. Choose effective exceptional $E$ with $-E$ relatively ample. For small $e>0$, $M_W+\phi^*A-eE$ is ample. A sufficiently diluted general member $G$ makes $(W,B_W+eE+G)$ sub-klt. The effective divisor $\Omega=B_Y+\phi_*G$ has this crepant pullback and satisfies $\Omega\sim_\Q B_Y+M_Y+A$.
\end{proof}

We summarize the genus-zero formula and the formula for surface fibres over a curve in a form convenient for the global argument.

\begin{theorem}[Threshold canonical bundle formulas]
\label{a-thm:cbf}
Fix $\epsilon>0$ and a finite rational seed $R$. For $p>p_0(\epsilon,R)$, let $f:(X,B)\to Z$ be a contraction from a projective lc threefold pair to a projective normal base $Z$, with $K_X+B\sim_\Q f^*L_Z$ and $B$ big over $Z$. Suppose either
\begin{enumerate}[label=(\roman*)]
\item $\dim Z=2$ and the geometric generic log fibre is lc with smooth genus-zero underlying curve, and its boundary has the decomposition into tame marked points and a residual cycle of \cref{cbf:anchored}; or
\item $Z$ is a smooth curve, the underlying generic surface is $\epsilon$-lc and Fano, and $\coeff(B^h)\subset\Phi(R)$.
\end{enumerate}
Then
\[
K_X+B\sim_\Q f^*(K_Z+B_Z+M_Z),
\qquad
B_Z=\sum_P\bigl(1-\lct_{\eta_P}(X,B;f^*P)\bigr)P,
\]
and the threshold moduli b-divisor is b-nef and b-$\Q$-Cartier.
\end{theorem}
\begin{proof}
Take $p_0$ larger than $p_{\rm sf}(\epsilon,R)$, five, and $2/\delta_R$, where $\delta_R=\min(\Phi(R)\cap(0,1])$. Case (i) is \cref{cbf:anchored}. In case (ii), \cref{a-lem:surface-fibre} gives geometric lc and a log-smooth generic resolution. Apply \cref{prop:controlled-gglc-dlt} to choose an effective crepant $\Q$-factorial dlt model whose geometric generic pair is dlt. The resulting fibration is geometrically log canonical at the generic point, and its generic pair is $\Q$-factorial dlt. Benozzo's \cite[Theorem 7.10 and Corollary 7.11]{Benozzo} gives nefness of its threshold moduli divisor on $Z$. Crepancy preserves every fibre threshold, since both log discrepancies and orders of the pulled-back base parameters agree. Hence the same moduli divisor works on the original model. A normal curve has no nontrivial proper birational modification, so nefness here is b-nefness and the divisor is b-$\Q$-Cartier.
\end{proof}

\subsection{A wild counterexample over a surface}\label{sec:wild-counterexample}

For each prime $p>5$ we construct a smooth $\PP^1$-bundle $f:X\to Y$ over a smooth projective surface and an effective klt boundary $\Gamma$ of relative degree two. The general log fibre is klt, but the moduli divisor determined by the fibre thresholds has negative degree on a curve in $Y$.

Start with $Z=\PP^1_u\times\PP^1_s$ and $X_0=\PP^1_x\times Z$. Write $P=(u=0)$, $Q=(s=\infty)$ on $Z$ and $S=(x=0)$ on $X_0$. Let $D$ be the closure of $u^p(x^p-x)=s$. It has multidegree $(p,p,1)$ in $(x,u,s)$. The generic Artin--Schreier equation $x^p-x=s/u^p$ is irreducible since its right side has a simple pole at $s=\infty$. Put
\[
 b=\frac{3}{2p+1},\qquad a=pb,\qquad c=2-a,\qquad
 \delta=\frac{a-1}{p},\qquad \Gamma_0=bD+cS.
\]
The $p$ conjugate points on the geometric generic fibre have total weight $a>1$. Together with the section of weight $c$ they give a klt boundary of degree two.

Blow up $Z$ at $(0,1),(0,2),(0,3)$ and write $\pi:Y\to Z$ for the resulting map, with exceptional curves $F_i$. Take $X=\PP^1_x\times Y$, with projection $f:X\to Y$, and let $g=\mathrm{id}_{\PP^1}\times\pi:X\to X_0$. Neither $D$ nor $S$ contains a whole fibre over a blowup centre. Thus $\Gamma=g^*\Gamma_0$ consists of their strict transforms with the same coefficients. The multidegree of $D$ gives
\[
 K_X+\Gamma\sim_\Q f^*\bigl(K_Y+\pi^*(aP+bQ)\bigr).
\]

\emph{Singularities of the total pair.} Near $u=0$, $z=1/x=0$ the equation of $D$ is $u^p(1-z^{p-1})=sz^p$. Blowing up $(u,z)$ gives exceptional coefficient $a-1<1$. For finite $s$ the strict transform is $w^p(1-z^{p-1})=s$ in the chart $u=zw$; at $s=\infty$ it is $r(1-u^{p-1}w^{p-1})=w^p$ with $r=1/s$ and $z=uw$. These equations give transverse smooth divisors along the exceptional intersection, and $S$ misses the centre. Elsewhere $D$ is smooth except at $u=s=\infty$, $x\in\mathbb F_p$, where its equation in \'etale coordinates is $ry=v^p$, with $v=1/u$ and $y=x^p-x$. This is an $A_{p-1}$ surface singularity, so $(X_0,D)$ is lc there by inversion of adjunction. Since $b+c<1$ and $S$ is smooth, convexity of log discrepancies proves that $(X_0,\Gamma_0)$ is klt.

For each blowup centre $(0,d)$ take the smooth surface $T=(s=d)$ in $X_0$. Its restricted curve $D|_T$ is smooth except at $u=z=0$. There, writing $\lambda^p=d$ and $w=u-\lambda z$, its equation is $w^p(1-z^{p-1})=d z^{2p-1}$. Its Newton edge has endpoints $(p,0),(0,2p-1)$ in $(w,z)$, and the edge binomial has nonzero $z$-derivative on the torus. A regular subdivision thus computes the threshold as the minimum of $(i+j)/\min\{p\,i,(2p-1)j\}$ over positive weights. This is $1/p+1/(2p-1)>b$. Since $S$ misses the cusp and $b+c<1$, the surface pair $(T,h\Gamma_0|_T)$ is klt for rational $h>1$ sufficiently close to one. By \cref{a-thm:ioa}, $(X_0,h\Gamma_0+T)$ is lc near $T$. Over a neighbourhood of $(0,d)$,
\[
 K_X+h\Gamma+T'=g^*(K_{X_0}+h\Gamma_0+T),
\]
where $T'$ is the strict transform. The exceptional divisor in $K_X-g^*K_{X_0}$ also occurs with coefficient one in $g^*T$. Thus $(X,h\Gamma)$ is lc above the centre. Since $X$ is smooth and $h>1$, $(X,\Gamma)$ is klt there.

\emph{The discriminant.} Write $\widetilde P,\widetilde Q$ for the strict transforms and $t_A$ for the fibre threshold at the generic point of a prime $A$ on $Y$. Over $\widetilde P$ the blowup of $(u,z)$ above gives exceptional coefficient $a+t-1$ for fibre coefficient $t$. The strict transforms of $D$ and the fibre meet that exceptional divisor at distinct points, so $t_{\widetilde P}=2-a$. Over $\widetilde Q$, the smooth multisection has contact order $p$ with the fibre. Two smooth branches of coefficients $h_1,h_2\le1$ and contact order $p$ are lc precisely when $h_1+h_2\le1+1/p$, as follows by successive point blowups. Hence $t_{\widetilde Q}=1+1/p-b=1-\delta$.

At $\eta_{F_i}$ use $s=d+u\nu$, with residue field $k(\nu)$. Blow up $(w,z)$, where $w=u-\lambda z$ as above. In the chart $w=zv$ the strict transform of $D$ is
\[
 v^p(1-z^{p-1})-\nu z(v+\lambda)-d z^{p-1}=0 .
\]
It is smooth at $v=z=0$ and tangent to the exceptional curve $z=0$ to order $p$. The strict fibre meets that curve at $v=-\lambda$. The exceptional coefficient is $a+t-1$, so the tangency condition is $b+(a+t-1)\le1+1/p$. This gives $t_{F_i}=2-a-\delta$. At every other prime the multisection is generically \'etale over the base. Where one branch meets $S$, their restricted coefficients sum to $b+c<1$. Adjunction to the reduced fibre divisor therefore gives threshold one.

It follows that the discriminant and moduli part on $Y$ are
\[
 B_Y=(a-1)\pi^*P+\delta\left(\widetilde Q+\sum_i F_i\right),
 \qquad
 M_Y\sim_\Q \pi^*(P+Q/p)-\delta\sum_i F_i .
\]
Since $P^2=0$, $P\cdot Q=1$ and $F_i\cdot\widetilde P=1$, we obtain
\[
 M_Y\cdot\widetilde P=\frac1p-3\delta
       =\frac{4-p}{p(2p+1)}<0.
\]
The trace on a smooth surface of a nef divisor on a higher model has nonnegative degree on every curve, by the projection formula. Thus this moduli class on $Y$ admits no b-nef extension. The wild orbit has weight $a>1$, exceeding the residual weight allowed in \cref{cbf:anchored}. Its coefficient $b$ tends to zero as $p$ grows, whereas the positive elements of a fixed $\Phi(R)$ are bounded below away from zero.

\subsection{Fano type of the base}

\begin{corollary}
\label{a-cor:fano-base}
Let $(X,B)$ be a three-dimensional projective lc pair, let $f\colon X\to Z$ be a contraction, and assume
\[
\begin{gathered}
 K_X+B\sim_\Q0/Z,\qquad -(K_X+B)\text{ is nef},\\
 X\text{ is of Fano type},\qquad \coeff(B)\subset\Phi(R).
\end{gathered}
\]
Then $Z$ is of Fano type in either of the following cases:
\begin{enumerate}
 \item $Z$ is a projective normal surface and $p>p_1(R):=p_0(1/2,R)$;
 \item $Z$ is a projective normal curve and $p>5$.
\end{enumerate}
In the curve case even the hypotheses involving $B$ are unnecessary: any contraction from a projective Fano-type threefold to a normal projective curve has $Z\simeq\mathbb P^1$.
\end{corollary}

\begin{proof}
Suppose first that $Z$ is a surface. Since $-K_X$ is big and $B_\eta\sim_\Q-K_{X_\eta}$, the boundary is big over $Z$. Write $-(K_X+B)\sim_\Q f^*L$; nefness of the pullback implies that $L$ is nef. Choose a big effective klt boundary $\Delta$ with
\begin{equation}
 K_X+\Delta\sim_\Q0.
 \label{eq:display-3-12}
\end{equation}
For an ample $A$ on $Z$, Kodaira's lemma and openness of klt give an effective klt $\Delta'$ with $K_X+\Delta'\sim_\Q-a f^*A$, $a>0$. Explicitly, write $\Delta\sim_\Q G+H$ with $G\geq0$ and $H$ ample, choose $H_\zeta\sim_\Q H-\zeta f^*A$ general, and take
\begin{equation}
 \Delta'=(1-\delta)\Delta+\delta G+\delta H_\zeta,
 \qquad a=\delta\zeta,
 \label{eq:display-3-13}
\end{equation}
for small positive $\zeta,\delta$. Then
\begin{equation}
 K_X+\Delta'\sim_\Q-a f^*A.
 \label{eq:display-3-14}
\end{equation}

For small rational $0<t<1/2$, set
\begin{equation}
 \Gamma=(1-t)B+t\Delta'.
 \label{eq:display-3-15}
\end{equation}
This pair is klt and
\begin{equation}
 K_X+\Gamma\sim_\Q-f^*((1-t)L+taA).
 \label{eq:display-3-16}
\end{equation}
By \cref{a-lem:curve-fibre}, the original geometric generic pair is lc and its marked points are separable. At each coefficient-one point, $\Delta'_{\bar F}$ has coefficient below one; at every other point the coefficient of $B_{\bar F}$ is below one. The finite support therefore makes $(\bar F,\Gamma_{\bar F})$ klt for sufficiently small $t$.

The original marked points number at most $2/\delta_R<p$, so their Galois splitting group is a subgroup of $\mathfrak S_N$ of order prime to $p$. Absorb into their weights the components of $t\Delta'$ at those points. The residual cycle has total geometric weight at most $2t<1$. Apply \cref{cbf:anchored} to obtain
\begin{equation}
 K_X+\Gamma\sim_\Q f^*(K_Z+B_Z^\Gamma+M_Z^\Gamma).
 \label{eq:display-3-17}
\end{equation}
Klt of $(X,\Gamma)$ makes every divisorial fibre threshold positive. Effectivity also bounds the thresholds at original base primes by one, so $B_Z^\Gamma\geq0$. On higher models only positivity is needed: it gives generalized klt even when a crepant vertical coefficient is negative and its threshold exceeds one.

Put $q=ta$. For $0<\eta<1$, \cref{a-lem:ordinary-approx} gives a klt boundary $\Omega\geq0$ with
\begin{equation}
 \Omega\sim_\Q B_Z^\Gamma+M_Z^\Gamma+(1-\eta)qA.
 \label{eq:display-3-18}
\end{equation}
Consequently
\begin{equation}
 -(K_Z+\Omega)\sim_\Q(1-t)L+\eta qA
 \label{eq:display-3-19}
\end{equation}
is ample, proving Fano type.

If $Z$ is a curve, $X$ is rationally chain connected by \cite[Theorem~1.5, equivalently Proposition~4.11]{GNT}. Its image $Z$ is therefore a rationally chain connected normal projective curve, hence $Z\simeq\PP^1$.
\end{proof}

\subsection{Effective discriminants and moduli indices}

\begin{lemma}
\label{moduli:tree}
Let $R$ be a discrete valuation ring with fraction field $K$. Suppose a log Calabi--Yau generic fibre is a split pair $(\mathbb P^1_K,\sum b_iP_i)$ with distinct $K$-points, $0<b_i\leq1$, $\sum b_i=2$, and $qb_i\in\mathbb Z$. On a crepant model over $R$, allow the vertical boundary to have negative coefficients. Fix the moduli representative by a rational function trivializing $q(K+B)$ on the generic fibre. Then the coefficient of $qM$ at the closed point of $\Spec R$ is integral.
\end{lemma}
\begin{proof}
Start with $\mathbb P^1_R$ and the closures of the marked points. Successively blow up their common specializations until the sections are disjoint. Every centre is a rational smooth point of a component of the closed fibre. Indeed, after blowing up such a point the sections meet the new exceptional curve away from its intersection with the old fibre; in coordinates their slopes are finite. Repetition decreases the finite intersection multiplicities of distinct sections. Thus this process terminates and produces a regular model $g\colon Y\to\Spec R$ whose closed fibre $F=\sum C_j$ is a reduced snc tree of rational curves, with disjoint sections avoiding its nodes.

Write the crepant boundary as $B_Y=\sum b_iH_i+\sum a_jC_j$. It is a sub-boundary, not necessarily effective. The generic log divisor is relatively trivial, so $(K_Y+B_Y)\cdot C_j=0$ for every component. At a vertex $j$ of valency $v_j$, put $w_j=\sum_{H_i\cap C_j\ne\varnothing}b_i$. The intersection equations are
\[
 0=-2+v_j+w_j-v_ja_j+\sum_{k\sim j}a_k.
\]
Summing these equations over either subtree cut off by an edge shows that the difference of the two adjacent $a$-coefficients is an integer plus a sum of the $b_i$. Consequently
\begin{equation}
 q(a_j-a_k)\in\mathbb Z\quad\text{for all }j,k.
 \label{moduli:eq-differences}
\end{equation}
Let $a_*:=\max_j a_j$. Because the pair is log smooth, its fibre threshold is $t=1-a_*$. Equation \eqref{moduli:eq-differences} implies that $q(B_Y+tF)$ is integral.

For the fixed generic trivialization $\alpha$, write
\[
 q(K_Y+B_Y)+\operatorname{div}(\alpha)=qg^*L_R.
\]
Adding $qtF$ makes the left side integral. The fibre is reduced, so $q(L_R+t[0])$ is integral on $\Spec R$. Since the discriminant coefficient is $1-t$, the identity $M_R=L_R-K_R-(1-t)[0]$ proves the assertion. The index does not depend on the number of blow-ups or on the multiplicities on a previous model.
\end{proof}

\begin{theorem}[Effective genus-zero formula]
\label{moduli:effective-surface-base}

Fix a finite rational set $R_0\subset[0,1]$. There are $q>0$, $p_0>5$, and a finite rational set $I\subset[0,1]$, depending only on $R_0$, with the following property. Let $f\colon(X,B)\to Z$ be an lc-trivial projective contraction from an lc threefold to a projective normal surface, with $B\in\Phi(R_0)$, $X$ of Fano type over $Z$, and $p>p_0$. Then its threshold discriminant and moduli b-divisor satisfy
\[
 B_Z\in\Phi(I),\qquad q\mathbf M\text{ is b-nef and b-Cartier},
 \qquad q(K_X+B)\sim qf^*(K_Z+B_Z+M_Z).
\]

\end{theorem}
\begin{proof}
Enlarge $R_0$ by $\{0,1\}$ at the outset, so that the coefficient calculation below also represents the boundary coefficients zero and one. \emph{Step 1. A fixed generic trivialization.} Use \cref{bridge:relative-complements} to choose a bounded index $q_1$. On the projective generic fibre an effective complement increment is numerically trivial and hence zero. Therefore $q_1(K_X+B)|_{X_\eta}\sim0$. Choose its rational trivialization $\alpha$. The divisor $q_1(K_X+B)+\operatorname{div}(\alpha)$ is vertical and relatively $\Q$-linearly trivial. It equals $q_1f^*L_Z$ for a $\Q$-Cartier divisor $L_Z$ on $Z$.

Indeed, a relatively $\Q$-linearly trivial vertical divisor descends: the rational function in a relative trivialization has no zero or pole on the projective generic fibre, so it belongs to $k(Z)^*$.

\emph{Step 2. Discriminant coefficients on the original base.} At a prime $P\subset Z$, form the lc threshold pair $B+t_Pf^*P$ over $\eta_P$. The dlt threshold construction, anti-MMP, and floor-preserving proof of \cref{bridge:relative-complements} produce a bounded $q_2$-complement retaining a vertical lc place over $\eta_P$. Enlarge $q_1,q_2$ to a common multiple, denoted $q_0$. The complement increment is an effective vertical relatively trivial divisor, so it is a multiple of $f^*P$ at $\eta_P$. Its coefficient is exactly $t_P$, by the retained lc place. Thus it is the saturated pair $B+t_Pf^*P$ there.

Choose a component of $f^*P$ of multiplicity $m$ and write its boundary coefficient $b=1-r/l$, with $r\in R_0$ and $l\geq1$. Its complement coefficient satisfies $b^+=b+t_Pm$ and $q_0b^+\in\mathbb Z$. If $b^+=1$, then $t_P=r/(lm)$. Otherwise $b\leq b^+\leq1-1/q_0$ forces $l\leq q_0$, and
\[
 t_P=\frac{s}{m},\qquad s=b^+-1+\frac rl.
\]
The possible $s\in[0,1]$ form a finite set determined by $R_0,q_0$. This gives $1-t_P\in\Phi(I)$ for a fixed finite set $I$.

\emph{Step 3. A bounded tame splitting cover.} Put $\delta=\min(\Phi(R_0)\cap(0,1])$ and $d=\lceil2/\delta\rceil$. \cref{a-lem:curve-fibre} shows that for $p>d$ the geometric generic fibre is smooth and the geometric generic log pair is lc. The degree identity $\deg B_{\bar\eta}=2$ bounds the number of its marked points by $d$ and excludes inseparable marked points. The splitting field of all marks is a Galois extension of degree dividing $d!$. It contains a point of the conic, so the conic is split as well. Write its group as $G$. Since $p>d$, the cover is tame at every divisorial valuation.

Apply \cref{a-lem:genus-zero-cbf} with no remaining cycle. Choose a model $Z_0$ on which the b-nef moduli divisor descends. Normalize $Z_0$ in this splitting field and take its minimal regular resolution $T$. Uniqueness of the minimal resolution of a normal surface lifts every element of $G$, giving a $G$-action on $T$. The quotient $W=T/G$ is normal and maps birationally to $Z_0$; $\tau\colon T\to W$ is finite. All these models are projective over the original base.

Define the boundary upstairs crepantly, allowing negative coefficients. At a divisorial valuation, the tame discrepancy formula and the multiplicity formula give $t_{P'}=e(P'/P)t_P$. Thus
\[
 K_T+B_T=\tau^*(K_W+B_W),\qquad M_T=\tau^*M_W.
\]
These equalities hold for threshold b-divisors and therefore on every higher model. They are the tame specialization of the descent in \cref{a-lem:genus-zero-cbf}.

\emph{Step 4. Integral moduli on the split smooth model.} Fix an arbitrary prime of $T$ and localize at its DVR. The generic marked conic is split. The fixed generic trivialization of $q_0(K_X+B)$ pulls back; in particular $q_0b_i$ is integral for every horizontal coefficient. Apply \cref{moduli:tree}. It proves integrality of the coefficient of $q_0M_T$ at this prime. Since the prime was arbitrary and $T$ is smooth, $q_0M_T$ is Cartier.

\emph{Step 5. Descent of a Cartier multiple.} The semilocal norm argument in the proof of \cref{cbf:anchored} gives $|G|q_0M_W$ Cartier. Take $q=d!q_0$. Since $W$ dominates a model on which $\mathbf M$ descends, $q\mathbf M$ is b-Cartier. Nefness follows from the canonical bundle formula, and the fixed generic trivialization gives the asserted $q$-linear equivalence.
\end{proof}

\begin{proposition}[Effective formula over a curve]
\label{moduli:effective-curve-base}
Fix a finite rational seed $R_0\subset[0,1]$. There are a positive integer $q=q(R_0)$ and a finite rational set $I=I(R_0)$ with the following property. Let $f\colon(X,B)\to C$ be an lc-trivial projective contraction from an lc threefold to a projective normal curve in characteristic $p>5$, where $X$ is of Fano type over $C$ and $B\in\Phi(R_0)$. Suppose its threshold canonical bundle formula has nef moduli divisor. Then
\[
 B_C\in\Phi(I),\qquad qM_C\text{ is Cartier},\qquad
 q(K_X+B)\sim qf^*(K_C+B_C+M_C).
\]
\end{proposition}
\begin{proof}
Steps 1 and 2 of \cref{moduli:effective-surface-base} apply using the curve-base part of \cref{bridge:relative-complements}. They give the generic trivialization $q(K_X+B)+\operatorname{div}(\alpha)=qf^*L_C$ and, near each closed point $c$, a $q$-complement $B+t_cf^*c$ with a vertical lc place. Let $\beta$ trivialize this complement. Subtraction gives
\[
 qf^*(L_C+t_cc)+\operatorname{div}(\beta/\alpha)=0.
\]
The divisor of $\beta/\alpha$ is vertical, hence $\beta/\alpha\in k(C)^*$. Injectivity of pullback on rational Weil divisors gives a principal equality on $C$. It follows that $q(L_C+t_cc)$ is integral at $c$, and therefore that $qM_C=q(L_C+t_cc)-q(K_C+c)$ is integral there. A normal curve is regular, so $qM_C$ is Cartier. The coefficient calculation is exactly Step 2 above.
\end{proof}

\subsection{Lifting a complement from the base}

Let $f:X\to V$ be an lc-trivial threefold fibration to a surface, let $V'\to V$ be a higher model, and put $G_V=B_V^+-B_V\geq0$ and $G=f^*G_V$. We use the effective formula of \cref{moduli:effective-surface-base}.

\begin{proposition}
\label{r-prop:valuation-lift}
Let $f\colon(X,B)\to V$ be an lc-trivial contraction from a threefold to a normal $\Q$-factorial surface. Suppose its threshold canonical bundle formula gives a rational trivialization
\[
 q(K_X+B)+\Div(\alpha)=qf^*(K_V+B_V+M_V).
\]
Assume $(V,B_V+\mathbf M)$ has a generalized monotonic $n$-complement, with $q\mid n$. Then $(X,B)$ has a monotonic $n$-complement.
\end{proposition}
\begin{proof}
Put $G_V=B_V^+-B_V\geq0$ and $B^+=B+f^*G_V$. The given principal equality and the base complement give
\begin{equation}
 n(K_X+B^+)\sim0.
 \label{r-eq:lift-principal}
\end{equation}
We check lc at a vertical prime $E$ over $X$. Its restriction $w$ to $k(V)$ is a nontrivial rank-one discrete valuation. The residue-field inequality and Abhyankar's inequality give
\[
 2=\operatorname{trdeg}_k\kappa(v_E)
   \leq\operatorname{trdeg}_k\kappa(w)+1\leq2.
\]
Thus $w$ is divisorial. On a higher model $V'$, write it as a multiple of $\ord_D$, and choose a compatible model $X'$. Let $m_E=v_E(f'^*D)>0$, let $t_D$ be the fibre threshold, and let $g_D$ be the coefficient of the pullback of $G_V$. Generalized lc of the base complement and the threshold definition give
\[
 g_D\leq t_D,\qquad a(E,X,B)\geq m_Et_D.
\]
Hence the discrepancy after adding $f^*G_V$ is nonnegative. Horizontal discrepancies are unchanged, so $(X,B^+)$ is lc. Finally, \eqref{r-eq:lift-principal} makes $nB^+$ integral, and monotonicity then gives the complement inequality.
\end{proof}

\subsection{Global fibered complements}

\begin{proposition}
\label{global:surface-base}

Fix a finite rational set $R\subset[0,1]$. There are $n(R)$ and $p_0(R)$ such that a projective lc threefold $(X,B)$ in characteristic $p>p_0(R)$ has a monotonic $n(R)$-complement whenever $X$ is of Fano type, $B\in\Phi(R)$, $-(K_X+B)$ is nef, and there is a contraction $X\to V$ with $\dim V=2$ and $K_X+B\sim_{\Q,V}0$.
\end{proposition}
\begin{proof}
Apply \cref{moduli:effective-surface-base}. It gives a bounded $q=q(R)$, $I=I(R)$, a $q$-fold formula, and a b-nef moduli part with $q\mathbf M$ b-Cartier. The base is of Fano type by \cref{a-cor:fano-base}. The threshold definition makes $(V,B_V,\mathbf M)$ generalized lc on every model, and its negative log divisor is nef: this is detected by its nef pullback using the projection formula. Apply \cref{gen:global-surface}. The valuation argument in \cref{r-prop:valuation-lift} pulls back the complement with the same index, including divisors mapping to points on the original surface.
\end{proof}

\begin{lemma}
\label{main:25}
For a fixed DCC coefficient set, a dlt Fano-type surface log Calabi--Yau pair over the generic field of a curve has a uniform gap between zero and its positive log discrepancies, in characteristic different from $2,3,5$.
\end{lemma}
\begin{proof}
If positive discrepancies tended to zero, extract the corresponding divisors. Their crepant coefficients would form a strictly increasing sequence tending to one. The extraction is of Fano type: combine its effective lc log Calabi--Yau boundary with the crepant sub-boundary of a klt log Fano pair, taking the coefficient on the former sufficiently close to one. The resulting boundary is effective klt with nef and big negative log divisor. The anticanonical divisor, and hence the log Calabi--Yau boundary, is big. Waldron's surface global ACC over the possibly imperfect generic field \cite[Theorem 1.4]{Waldron} then contradicts the infinitely many extracted coefficients.
\end{proof}

\begin{lemma}
\label{global:generic-improvement}
Fix a finite rational seed $R$. Let $(X,B)\to C$ be an lc-trivial contraction from a projective lc threefold of Fano type to a curve, with $B\in\Phi(R)$ and $-(K_X+B)$ nef. There is a crepant birational model $(Y,B_Y)$, still of Fano type with the same hyperstandard seed enlarged by one, such that the underlying generic surface $Y_\eta$ is $\epsilon(R)$-lc.
\end{lemma}
\begin{proof}
Take a $\Q$-factorial dlt model of $(X,B)$; the usual convex crepant-boundary argument proves its global Fano type. Work from this model. \Cref{main:25}, applied over $k(C)$, gives a uniform $\delta>0$ such that a log discrepancy of $(X_\eta,B_\eta)$ is either zero or at least $\delta$. Choose $0<\epsilon<\min\{\delta,1\}$, depending only on $R$.

There are only finitely many divisorial valuations of the klt surface $X_\eta$ with
\[
 a(E,X_\eta,0)\leq\epsilon.
\]
To see this, resolve the underlying surface with snc exceptional divisor. All crepant log discrepancies are positive. Any further blowup at a point on just one exceptional curve has discrepancy greater than one. Blowups over a crossing with discrepancies $a_1,a_2>0$ have toroidal discrepancies $u a_1+v a_2$, with positive integers $u,v$. Only finitely many can be at most $\epsilon<1$. This argument is on regular surfaces and does not assume perfection of $k(C)$. Denote the resulting finite set by $\mathcal E$. Each is an lc place of the generic log pair, since
\[
 0\leq a(E,X_\eta,B_\eta)
 \leq a(E,X_\eta,0)\leq\epsilon<\delta.
\]
Their closures define horizontal divisorial lc places of $(X,B)$.

Extract exactly these primes. We construct an extraction of this finite set as follows. For small $s>0$, the pair $(X,(1-s)B)$ is klt and each specified prime has discrepancy $s\,a(E,X,0)<1$. On a log resolution containing $\mathcal E$, keep their positive crepant coefficients, and raise all the other exceptional coefficients to positive numbers strictly between their crepant coefficients and one. The resulting klt log divisor, relatively over $X$, is an effective exceptional divisor supported exactly on the unwanted primes. Its relative threefold MMP terminates by the klt minimal-model theorem. The specified primes survive because they are not in that support; negativity contracts every unwanted prime. The endpoint is a $\Q$-factorial klt extraction $h:Y\to X$ of exactly $\mathcal E$. Restoring the crepant boundary gives
\[
 K_Y+B_Y=h^*(K_X+B),\qquad B_Y=B^\sim+\sum_{E\in\mathcal E}E.
\]
Thus its coefficients remain in $\Phi(R)\cup\{1\}$, its negative log divisor stays nef, and the convex-combination argument with a crepant klt log Fano subboundary proves Fano type of $Y$.

Finally,
\[
 K_Y-h^*K_X=\sum_{E\in\mathcal E}(a(E,X,0)-1)E\leq0.
\]
Every nonextracted generic discrepancy of the underlying variety therefore increases and was already greater than $\epsilon$. The extracted primes now have discrepancy one. Hence $Y_\eta$ is $\epsilon$-lc.
\end{proof}

\begin{proof}[Proof of \cref{global:assembly}]
The surface base case is \cref{global:surface-base}. If the base is a curve $C$, then $C=\PP^1$ by \cref{a-cor:fano-base}. Apply \cref{global:generic-improvement} and run a $K_Y$-MMP over $C$. Fano type and the generic underlying $\epsilon(R)$-lc condition are preserved. Every birational step is crepant for $K_Y+B_Y$, which is pulled back from $C$; thus lc, nefness and the coefficient set are preserved.

The endpoint is a Mori fibre space $X'\to Z/C$. If $\dim Z=2$, apply \cref{global:surface-base}. Otherwise $Z=C$, and the generic fibre is an $\epsilon(R)$-lc Fano surface with $B'_\eta\sim_\Q-K_{X'_\eta}$. The curve base case of \cref{a-thm:cbf} applies for a characteristic bound depending only on $R$. By \cref{moduli:effective-curve-base}, the discriminant is hyperstandard with fixed seed and $qM_C$ is Cartier for bounded $q$. Apply \cref{gen:curve} to the generalized lc pair on $C$.

To lift its complement, a vertical divisorial valuation restricts to $m\ord_c$ on $k(C)$. If $g_c$ is the complement increment and $t_c$ the fibre threshold, then
\[
 g_c\leq t_c,\qquad a(E,X',B')\geq mt_c\geq mg_c.
\]
Thus the lifted pair is lc; horizontal discrepancies are unchanged. The effective formula preserves the principal complement multiple. Reverse the crepant MMP and push through the selective extraction. A common multiple of the curve and surface indices and the maximum of their characteristic bounds depend only on $R$.
\end{proof}

\begin{proof}[Proof of \cref{fib:nonbig}]
Choose a klt log Fano boundary $\Gamma$ on $X$ and a positive Cartier multiple $L$ of $N$. The divisor $L-(K_X+\Gamma)$ is ample, and $\nu(L)>0$. Thus \cite[Corollary~1.5]{ABL} makes a positive multiple of $L$ base-point-free. Its morphism, followed by Stein factorization, gives a contraction $f:X\to Z$ with $N\sim_\Q f^*A$ for an ample rational Cartier divisor $A$. The projection formula gives $\dim Z=\nu(N)\in\{1,2\}$. We can now apply \cref{global:assembly} with this $f$.
\end{proof}

\end{document}